\documentclass[article,12pt]{amsart}

\usepackage{amsmath}
\usepackage{amsfonts}
\usepackage{amssymb}
\usepackage{graphicx}
\usepackage{enumitem}
\usepackage{color}
\usepackage{array}
\usepackage{fullpage}
\usepackage{bbm}
\usepackage[round]{natbib}
\usepackage{MnSymbol}
\usepackage{url}

\newtheorem{thm}{Theorem}[section]

\newtheorem{lem}[thm]{Lemma}
\newtheorem{prop}[thm]{Proposition}

\theoremstyle{remark}
\newtheorem{rem}{Remark}[section]

\newcommand{\veps}{\varepsilon}
\newcommand{\wh}{\widehat}
\newcommand{\wt}{\widetilde}
\newcommand{\whn}{\wh{\theta}_{n}}
\newcommand{\wtn}{\wt{\theta}_{n}}
\newcommand{\whvn}{\wh{\vartheta}_n}
\newcommand{\hsp}{\hspace{0.5cm}}
\newcommand{\T}{^{\, T}}
\newcommand{\diag}{\textnormal{diag}}

\newcommand{\eqdef}{\, \overset{\Delta}{=}\, }

\newcommand{\ind}{\mathbbm{1}}

\newcommand{\correc}[1]{\textcolor{black}{#1}}

\newcommand{\dE}{\mathbb{E}}
\newcommand{\dP}{\mathbb{P}}
\newcommand{\dR}{\mathbb{R}}

\newcommand{\cA}{\mathcal{A}}
\newcommand{\cC}{\mathcal{C}}
\newcommand{\cF}{\mathcal{F}}
\newcommand{\cH}{\mathcal{H}}

\newcommand{\cN}{\mathcal{N}}

\newcommand{\cvgp}{~ \overset{p}{\longrightarrow} ~}
\newcommand{\cvgd}{~ \overset{d}{\longrightarrow} ~}
\newcommand{\limn}{\lim_{n\, \rightarrow\, +\infty}}

\def\leq{\leqslant}
\def\geq{\geqslant}

 \numberwithin{equation}{section}
\allowdisplaybreaks

\begin{document}

\title[Extent of instability in nearly-unstable processes]{Testing for the extent of instability in nearly-unstable processes}

\author{Marie Badreau}
\address{Laboratoire Manceau de Mathématiques, Le Mans Université, Avenue Olivier Messiaen, 72085 LE MANS Cedex 09, France.}
\email{marie.badreau.etu@univ-lemans.fr}

\author{Fr\'ed\'eric Pro\"ia}
\address{Univ Angers, CNRS, LAREMA, SFR MATHSTIC, F-49000 Angers, France.}
\email{frederic.proia@univ-angers.fr}

\begin{abstract}
This paper deals with unit root issues in time series analysis. It has been known for a long time that unit root tests may be flawed when a series although stationary has a root close to unity. That motivated recent papers dedicated to autoregressive processes where the bridge between stability and instability is expressed by means of time-varying coefficients. The process we consider has a companion matrix $A_{n}$ with spectral radius $\rho(A_{n}) < 1$ satisfying $\rho(A_{n}) \rightarrow 1$, a situation described as `nearly-unstable'. The question we investigate is: given an observed path supposed to come from a nearly-unstable process, is it possible to test for the `extent of instability', \textit{i.e.} to test how close we are to the unit root? In this regard, we develop a strategy to evaluate $\alpha$ and to test for $\cH_0 : ``\alpha = \alpha_0"$ against $\cH_1 : ``\alpha > \alpha_0"$ when $\rho(A_{n})$ lies in an inner $O(n^{-\alpha})$-neighborhood of the unity, for some $0 < \alpha < 1$. Empirical evidence is given about the advantages of the flexibility induced by such a procedure compared to the common unit root tests. We also build a symmetric procedure for the usually left out situation where the dominant root lies around $-1$.
\end{abstract}

\maketitle

\begin{center}
\textbf{Keywords}: Nearly unstable autoregressive process, Unit root test, Martingales.
\textbf{Mathematics Subject Classification}: 62M10 - 62F12 - 62F03 - 60G52 - 60G42.
\end{center}

\section{Introduction and Motivation}
\label{SecIntro}

This paper is devoted to the well-known unit root issue, still at the heart of many concerns in finance, econometrics, and more generally in time series analysis, at the origin, when undetected, of spurious regressions through unreliable estimates. In order to investigate the stochastic non-stationarity of an observed autoregressive path, two strategies are frequently encountered. The first one consists in estimating the coefficient behind the (supposed) unit root, and the second one directly looks for a hidden random walk in the disturbance. In a very particular case, \citet{DickeyFuller79} first tackled the unit root issue by considering an AR$(1)$ model with a linear trend and concluded that, when there is a unit root in the generating process, the least-squares estimator of the autoregressive parameter has an asymmetric distribution that can be written asymptotically as a functional of a family of detrended Wiener processes, actually depending on the trend. Such a result had been anticipated in the Gaussian case by \citet{White58} in the absence of trend and, among numerous other studies, \citet{ChanWei88} improved the assumptions on the noise by considering a sequence of martingale differences (with sufficient moments). The behavior of the associated $t$-statistic has been tabulated by \citet{DickeyFuller79} for testing purposes, see also \citet{MacKinnon91}. The \textit{ADF test} (Augmented Dickey-Fuller) is probably the most commonly used nowadays to evaluate the presence of a unit root in a causal ARMA$(p, q)$ process. It was formalized by \citet{DickeySaid81} when $p$ and $q$ are known, and by \citet{DickeySaid84} under a truncated approximation of the AR$(\infty)$ expression, therefore working for any $p$ and $q$. At the basis of their study is the following differentiated form of an AR$(k)$ process with parameter $\theta$,
\begin{equation}
\label{ModARDiff}
\forall\, n \geq 1, \hsp \Delta X_n = (\theta_0 - 1) X_{n-1} + \sum_{i=1}^{k-1} \delta_i\, \Delta X_{n-i} + \veps_n
\end{equation}
where \correc{$(\veps_n)$ is a white noise sequence with finite variance}, $\delta_i = -(\theta_{i+1} + \hdots + \theta_k)$ for $1 \leq i \leq k-1$ and $\theta_0 = \theta_1 + \hdots + \theta_k$, so that $(\Delta X_n)$ is stationary under the null hypothesis of unit root $\cH_0 : ``\theta_0=1"$ (actually trend-stationary if a trend is added in the model). That was an inspiration for us, and we will find a similar reasoning in this paper. The asymptotic behavior of the estimate is still the one of the seminal work of \citet{DickeyFuller79} as long as $k = O(n^{1/3})$ but the power of the test is impacted and, as it is explained by \citet{Schwert89} or by \citet{NgPerron95}, some distortions may occur for badly truncated processes. Let us also mention the non-parametric approach suggested by \citet{Phillips87b} and deepened by \citet{PhillipsPerron88} in which all correlation phenomenon are postponed to the noise of the model, which is now supposed to be strongly mixing. Overall there is an abundant literature on unit roots tests, we have just summarized here some important topics. The interested reader will find in-depth studies for example in \citet{Bhargava86}, \citet{Perron88}, \citet{SchmidtPhillips92}, etc. From a practical point of view, \citet{NelsonPlosser82} had already highlighted the presence of unit roots in a number of macroeconomic series \textit{via} the Dickey-Fuller strategy. Although not directly related to this paper, we found it useful to mention the complementary approach consisting in testing for trend-stationarity instead of integration. The noise is now assumed to include a random walk whose variance is zero under the null hypothesis and among such procedures, the most famous seem to be the \textit{KPSS test} of \citet{KwiatkowskiEtAl92} and the quite similar \textit{LMC test} of \citet{LeybourneMcCabe94}. Here again there is an ample literature but too far from our scope to dwell on it, let us instead return to unit roots. 

\smallskip

In the context of autoregressive processes, the substantial works of \citet{LaiWei83} and \citet{ChanWei88} helped to establish that the least-squares estimator is strongly consistent wherever the characteristic roots lie, but with very different convergence rates and limit distributions whether the process is stationary (stable), integrated (unstable), or even explosive. However, \citet{DeJongEtAl92} observed that the unit root tests may be empirically less powerful than their stable counterparts when, although very close to it, the spectral radius remains less than 1, and some econometricians are not fully convinced by the compatibility between the theoretical impacts of this discontinuity -- from $n^{1/2}$-normality to $n-$asymmetry while the spectral radius continuously moves towards the unit root -- and real-word data. These are strong arguments in favor of attempts to bridge stability and instability that motivated numerous studies devoted to intermediate models, including ours. Most of them involve a first-order process with time-varying coefficients $(X_{n,\, k})$, generated by the triangular form
\begin{equation}
\label{Intro_NewMod}
\forall\, n \geq 1,\, \forall\, 1 \leq k \leq n, \hsp X_{n,\, k} = \theta_n\, X_{n,\, k-1} + \veps_k.
\end{equation}
This is an easy-to-use relevant case to understand the underlying dynamics since to focus on the inner neighborhood of the unit root, we may only consider a sequence of coefficients satisfying $\vert \theta_n \vert < 1$ for all $n \geq 1$ but $\vert \theta_n \vert \rightarrow 1$, so that the model is always stable but increasingly close to instability. \citet{ChanWei87}, \citet{Phillips87a}, \citet{GiraitisPhillips06} and later \citet{PhillipsMagdalinos07} were interested in such a model and, among other conclusions, it was established that the least-squares estimator is asymptotically normal around $\theta_n$ with rate $(n v_n)^{1/2}$ when $\theta_n = 1 - c/v_n$, for any $1 \ll v_n \ll n$ and $c > 0$. The junction is not perfect -- there is still a discontinuity in the asymptotic variance when $v_n=1$ and in the asymptotic distribution when $v_n=n$ -- but these are promising results. Indeed, by setting $v_n = n^{\alpha}$ for some $0 < \alpha < 1$, the gap between stability and instability is filled with rates of the form $n^{1/2} \ll n^{(1+\alpha)/2} \ll n$. The studies of \citet{VanDerMeerEtAl99}, \citet{Park03}, \citet{BuchmannChan13}, or the more recent ones of \citet{Proia20} and \citet{BadreauProia2023} aimed at finding a kind of generalization to nearly-unstable AR$(p)$ processes. In the latter, the process $(\Phi_{n,\, k})$ is assumed to follow the VAR$_p(1)$ triangular form
\begin{equation}
\label{ModVar}
\forall\, n \geq 1,\, \forall\, 1 \leq k \leq n, \hsp \Phi_{n,\, k} = A_{n}\, \Phi_{n,\, k-1} + E_k
\end{equation}
where $(E_k)$ is a $p$-vectorial noise and
\begin{equation}
\label{CompMat}
A_{n} = \begin{pmatrix}
\theta_{n,\, 1} & \theta_{n,\, 2} & \hdots & \theta_{n,\, p} \\
\ddots & & & 0 \\
& I_{p-1} & & \vdots \\
& & \ddots & 0
\end{pmatrix}
\end{equation}
is the associated $p \times p$ companion matrix with spectral radius satisfying $\rho(A_n) < 1$ for all $n \geq 1$ and $\rho(A_n) \rightarrow 1$. Provided suitable assumptions of moments, the least-squares estimator is still shown to be consistent and asymptotically normal around $\theta_n$ with mixed rates (see Thm. 3 and 4), and it is deduced as a corollary that, when $\rho(A_n) = 1 - c/v_n$,
\begin{equation}
\label{NormCombLin}
\sqrt{n\, v_n}\, \big\langle L_n,\, \whn - \theta_n \big\rangle \cvgd \cN\big( 0,\, h_0^2 \big) \hsp \text{with} \hsp L_n\T = \left( \lambda_1^{i-1} \rho(A_n)^{1-i} \right)_{1\, \leq\, i\, \leq\, p}
\end{equation}
where $\lambda_1 = \pm 1$ (depending on the sign of the unit root in the unstable limit process) and the asymptotic variance $h_0^2 > 0$ is given as a function of $c$ and the eigenvalues of $A$, the limit matrix of $A_n$. However, although this result and the ones that may be found in the references above are very useful for understanding the behavior of the process and the estimator in the inner neighborhood of the unit circle, they definitely lack practical application. \correc{Beyond the scope of this study, the concept of near-instability has also been defined for first-order INAR processes (AR with integer values), see \textit{e.g.} \citet{IspanyEtAl03}. In that case, the mean value of the counting sequence is assumed to converge to one, that is to the unstable setting of an INAR(1) process, and some kind of bridges are built between the stable and unstable rates. An application to nearly-critical Galton-Watson processes can be found in the same reference or more deeply in \citet{IspanyEtAl05}. Let us also mention the study of \citet{BarretoChan24} related to nearly-unstable INARCH(1) processes, still dealing with counting time series but in a heteroscedastic context. The reader is referred to the large bibliography on the topic given in this very recent article.}

\smallskip

The question we wish to investigate in this paper is the following: given an observed path supposed to come from a nearly-unstable process (\textit{e.g.} doubtful ADF/KPSS diagnoses), is it possible to test for the `extent of instability', \textit{i.e.} to admit that the process is stable but to test how close we are to the unit root? Setting $\rho(A_n) = 1-c/v_n$ with $v_n = n^{\alpha}$, the question is clearly related to the choice of $\alpha$ and may reduce to a test of $\cH_0 : ``\alpha = \alpha_0"$ against $\cH_1 : ``\alpha > \alpha_0"$ for some test value $\alpha_0$. In addition, evaluating $\alpha$ is likely to provide adjustments in the confidence intervals associated with the autoregressive parameters. This is an innovative approach of the unit root issue (to the best of our knowledge), especially from a practical point of view. In Section \ref{SecTest}, we give some specific features of the model and the hypotheses before developing the test procedure just mentioned. We will see that, as a corollary of our reasonings, we obtain a version of the asymptotic normality for a corrected statistic closely related to the least-squares estimation that does not need a mixing of the rates matrix, actually deepening \citet[Thm. 4]{BadreauProia2023}. An empirical study is provided in Section \ref{SecEmp}, first on simulated data and then on real time series  -- the famous dataset coming from \citet{NelsonPlosser82} -- in order to assess the performances of our procedure and justify its utility. After a short conclusion in Section \ref{SecConclu}, we finally prove our results in Section \ref{SecProofs}.

\section{A (Quasi-)Unit Root Test}
\label{SecTest}

Let us consider the triangular model \eqref{ModVar}, where
\begin{equation*}
\forall\, n \geq 1,\, \forall\, 1 \leq k \leq n, \hsp \Phi\T_{n,\, k} = \begin{pmatrix} X_{n,\, k} & \cdots & X_{n,\, k-p+1} \end{pmatrix} \hsp \text{and} \hsp E\T_k = \begin{pmatrix} \veps_k & 0 & \cdots & 0 \end{pmatrix}.
\end{equation*}
The process $(\veps_k)$ is assumed to be a white noise with finite variance $\sigma^2 > 0$, and the arbitrary initial value $\Phi_{n,\, 0}$ is square-integrable and independent of the noise. The companion matrix $A_n$ of the process is given in \eqref{CompMat} and has spectrum $\{ \lambda_{n,\, i}\}_{1\, \leq\, i\, \leq\, p}$ where $\lambda_{n,\, 1}$ is the one with largest modulus (the complex eigenvalues are sorted according to their modulus, in descending order, with ties broken by lexicographic order, also in descending order). The following technical hypotheses are intended to reflect our understanding of `near instability'.

\begin{enumerate}[label=(\correc{A$\arabic*$})]
	\item \label{HypCA} \textit{Convergence of the companion matrix}. There exists a $p \times p$ matrix $A$ such that
	\begin{equation*}
		\limn A_n = A
	\end{equation*}
	with distinct eigenvalues $0\, <\, \vert \lambda_{p} \vert\, \leq\, \ldots\, \leq\, \vert \lambda_2 \vert\, \leq\, \vert \lambda_{1} \vert\, =\, \rho(A) = 1$, and the top-right element of $A$ is non-zero ($\theta_p \neq 0$).
	\item \label{HypUR} \textit{Number of unit roots}. There is exactly one unit root in $A$ ($\lambda_1 = \pm 1$ but $\vert \lambda_2 \vert < 1$).
	\item \label{HypSR} \textit{Spectral radius of the companion matrix}. The spectral radius of $A_n$ is given by
	\begin{equation*}
		\forall\, n \geq 1, \hsp \rho(A_n) = 1 - \frac{c}{n^{\alpha}}
	\end{equation*}
	for some $c > 0$ and $0 < \alpha < 1$.
\end{enumerate}

For any $d \geq 1$, the usual $d \times d$ matrices
\begin{equation}
	\label{MatK}
	I_d = \begin{pmatrix} 1 & 0 & \hdots \\
		0 & \ddots & \\
		\vdots & & 1
	\end{pmatrix} \hsp \text{and} \hsp K_d = \begin{pmatrix} 1 & 0 & \hdots \\
		0 & 0 & \\
		\vdots & & \ddots
	\end{pmatrix}
\end{equation}
will be frequently encountered and the first vector of the canonical basis of $\dR^{d}$ will be denoted by $e_d\T = (1,\, 0,\, \ldots,\, 0)$. For readability purposes, we prove the results as if the limit process had a positive unit root (that is, $\lambda_1 = 1$), so that $\lambda_{n,\, 1}$ is equal (to be rigorous, this is true only for $n$ sufficiently large) to the spectral radius $\rho_n = \rho(A_n)$ of the companion matrix. Nevertheless they remain valid for $\lambda_1 = -1$, at the cost of some adjustments that will be clarified alongside (see Section \ref{SecNegRoot}). Technically, we should assume in addition that $0 < c < n^{\alpha}$ to ensure that for all $n$, $0 < \rho_n < 1$. Since our results are asymptotic w.r.t. $n$, we will allow ourselves a slight abuse of notation and simply set $c > 0$. The parameters $c$ and $\alpha$ have a very different role to play: $c$ is a \textit{variance parameter} and is clearly less important than $\alpha$, which is a \textit{rate parameter}. \correc{In fact, the over-parametrization issued from the ratio $c/n^{\alpha}$ makes it impossible to dissociate $c$ from $\alpha$ without additional constraints, which is the reason why we first assume in Section \ref{SecTestKnown} that $c$ is known to focus on $\alpha$. However, a procedure will also be detailed in Section \ref{SecTestUnknown} for a joint evaluation of $c$ and $\alpha$, according to a certain empirical protocol.} First, like \citet{Park03}, it is convenient to express the process under the hierarchical form
\begin{equation}
\label{ModHiera}
\forall\, n \geq 1,\, \forall\, 1 \leq k \leq n, \hsp \left\{
\begin{array}{lcl}
    X_{n,\, k} & = & \lambda_{n,\, 1}\, X_{n,\, k-1} + V_{n,\, k} \\
	V_{n,\, k} & = & \beta_{n,\, 1}\, V_{n,\, k-1} + \ldots + \beta_{n,\, p-1}\, V_{n,\, k-p+1} + \veps_k
\end{array}
\right.    
\end{equation}
where we merely set $V_{n,\, k} = \veps_k$ if $p=1$. As it can be easily seen by factoring the autoregressive polynomial of $(X_{n,\, k})$ by $(1-\lambda_{n,\, 1} z)$, the companion matrix $\bar{A}_n$ associated with $(V_{n,\, k})$ has spectrum $\{ \lambda_{n,\, i} \}_{2\, \leq\, i\, \leq\, p}$. Under \ref{HypCA} and \ref{HypUR}, it is therefore a stable AR$(p-1)$ process with time-varying coefficients $\beta_{n,\, i} \rightarrow \beta_i$, such that $\bar{A}_n \rightarrow \bar{A}$ and $\rho(\bar{A}_n) \rightarrow \rho(\bar{A}) = \vert \lambda_2 \vert < 1$. So the limit process is itself stable and its (normalized) covariance matrix is given by
\begin{equation}
\label{DefSigma}
\Sigma_{p-1} = \sum_{k\, \geq\, 0} \bar{A}^{\, k}\, K_{p-1}\, (\bar{A}\T)^k.
\end{equation}
But that also leads to another representation of the process, closely related to \eqref{ModARDiff} mentioned in the introduction,
\begin{equation}
\label{ModDiff}
\forall\, n \geq 1,\, \forall\, 1 \leq k \leq n, \hsp X_{n,\, k} = \lambda_{n,\, 1}\, X_{n,\, k-1} + \sum_{i=1}^{p-1} \beta_{n,\, i} \left( X_{n,\, k-i} - \lambda_{n,\, 1}\, X_{n,\, k-i-1} \right) + \veps_k
\end{equation}
from which we will benefit in our following results. In consequence, it is straightforward to see that $\theta_n\T = ( \theta_{n,\,1},\, \ldots,\, \theta_{n,\,p})$ and $\beta_n\T = ( \beta_{n,\, 1},\, \ldots,\, \beta_{n, \,p-1} )$ are such that
\begin{equation}
\label{DefJn}
\forall n \geq 1, \hsp \theta_n = J_n \, \beta_n + \lambda_{n, \,1} \, e_p \hsp \text{where} \hsp
	J_n = 
	\begin{pmatrix}
		1 & &\\
		- \lambda_{n,\, 1} & \ddots &\\
		& \ddots &  1 \\
		& & - \lambda_{n,\, 1}
	\end{pmatrix} \hsp (= 0 ~\text{if $p=1$})
\end{equation}
is a rectangular matrix of size $p \times (p-1)$. To emphasize that the latent variables $V_{n,\, k}$ depend on $\alpha$, from now on we will use $V_{n,\, k}(\alpha)$ instead and we will do the same for all confusing notation. Accordingly, let
\begin{equation*}
\forall\, n \geq 1,\, \forall\, 1 \leq k \leq n, \hsp \Psi_{n,\, k}\T(\alpha) = \begin{pmatrix} X_{n,\, k} & V_{n,\, k}(\alpha) & \cdots & V_{n,\, k-p+2}(\alpha) \end{pmatrix}
\end{equation*}
where the arbitrary initial value $\Psi_{n,\, 0}(\alpha)$ is square-integrable, so that from \eqref{ModDiff},
\begin{equation}
\label{ModVarTheta}
\forall\, n \geq 1,\, \forall\, 1 \leq k \leq n, \hsp X_{n,\, k} = \vartheta_n\T(\alpha)\, \Psi_{n,\, k-1}(\alpha) + \veps_k
\end{equation}
where $\vartheta_n\T(\alpha) = ( \lambda_{n,\, 1},\, \beta_{n,\, 1},\, \ldots,\, \beta_{n, \,p-1} )$. Let also
\begin{equation}
\label{DefSnTn}
S_{n,\, n} = \sum_{k=0}^n \Phi_{n,\, k}\, \Phi_{n,\, k}\T \hsp \text{and} \hsp T_{n,\, n}(\alpha) = \sum_{k=0}^n \Psi_{n,\, k}(\alpha)\, \Psi_{n,\, k}\T(\alpha)
\end{equation}
together with the rate matrix of size $p \times p$, that shall help to control the convergence of $S_{n,\, n}$ and $T_{n,\, n}(\alpha)$, given by
\begin{equation}
\label{RateCn}
C_n(\alpha) = \diag\left( \frac{1}{1 - \rho_n(\alpha)},\, 1,\, \ldots,\, 1 \right).
\end{equation}
The sum of squares $\sum_k \Vert E_k \Vert^2$ in \eqref{ModVar} is obviously minimized by
\begin{equation}
\label{OLS}
\whn = S_{n,\, n-1}^{\, -1} \sum_{k=1}^n \Phi_{n,\, k-1}\, X_{n,\,k}
\end{equation}
whereas in \eqref{ModVarTheta}, it is minimized (as a function of $\alpha$), by
\begin{equation}
\label{OLS2}
\whvn(\alpha) = T_{n,\, n-1}^{\, -1}(\alpha) \sum_{k=1}^{n} \Psi_{n,\, k-1}(\alpha)\, X_{n,\, k}.
\end{equation}
If $p=1$, then $\whvn(\alpha)=\whvn$ coincides with $\whn$. But before coming to the study of the rate parameter $\alpha$, it should be noted that as an interesting consequence of our study, it is now possible to build a consistent and asymptotically normal estimator of $\theta_n$. On the basis of \eqref{DefJn}, let
\begin{equation}
\label{OLS3}
\wtn = \wh{J}_n\, \wh{\beta}_n + \wh{v}_n\, e_p
\end{equation}
where $\wh{J}_n$ is $J_n$ with $\lambda_{n,\, 1}$ that can be estimated by $\wh{v}_n = e_p\T\, \whvn$, the first component of $\whvn$, and $\wh{\beta}_n$ is $\whvn$ deprived of its first component (we drop the dependency on $\alpha$ in the notation of this paragraph, for readability). Then, we have the following result.
\begin{prop}
\label{PropCvgNormOLS3}
Assume that \ref{HypCA}, \ref{HypUR} and \ref{HypSR} hold, and that $\dE[\vert \veps_1 \vert^{\, 2+\nu}] = \eta_{\nu} < +\infty$ for some $\nu > 0$. Then,
\begin{equation}
\label{CvgNormOLS3}
\big\Vert \wtn - \theta_n \big\Vert \cvgp 0 \hsp \text{and} \hsp \sqrt{n}\, \big( \wtn - \theta_n \big) \cvgd \cN\big( 0,\, J\, \Sigma_{p-1}^{-1} J\T \big)
\end{equation}
where $\Sigma_{p-1}$ is given in \eqref{DefSigma} and
\begin{equation*}
J = \begin{pmatrix}
		1 & &\\
		-\lambda_1 & \ddots &\\
		& \ddots &  1 \\
		& & -\lambda_1
	\end{pmatrix} \hsp (= 0 ~\text{if $p=1$}).
\end{equation*}
\end{prop}
\begin{proof}
See Section \ref{SecProofCorOLS3}.
\end{proof}

The asymptotic normality is degenerate in the sense that $J\, \Sigma_{p-1}^{-1} J\T$ is not invertible (as a $p \times p$ matrix of rank at most $p-1$). This may be illustrated by the relation $v_p\T\, J = 0$ where $v_p\T = (\lambda_1^{i-1})_{1\, \leq\, i\, \leq\, p}$ so that \eqref{CvgNormOLS3} becomes
\begin{equation*}
\sqrt{n}\, v_p\T \big( \wtn - \theta_n \big) \cvgp 0
\end{equation*}
and a parallel can be made with \eqref{NormCombLin}. It is not surprising for $p=1$ since it is well-known that in that case, $\vert \wtn - \theta_n \vert = O_p(n^{-(1+\alpha)/2})$, see \textit{e.g.} \citet{PhillipsMagdalinos07}.

\begin{rem}
Strictly speaking, $\wtn$ -- as well as $\whn$ -- is a consistent estimation of $\theta$, the limit of $\theta_n$, for obviously $\Vert \wtn - \theta \Vert\, \leq\, \Vert \wtn - \theta_n \Vert + \Vert \theta_n - \theta \Vert$. But it can only be asymptotically normal under the stronger hypothesis $\Vert \theta_n - \theta \Vert = o_p(n^{-1/2})$, which requires in particular that $\alpha > 1/2$. In this case, an immediate consequence of \eqref{CvgNormOLS3} is
\begin{equation}
\label{NormOLS3}
\sqrt{n}\, \big( \wtn - \theta \big) \cvgd \cN\big( 0,\, J\, \Sigma_{p-1}^{-1} J\T \big).
\end{equation}
\end{rem}

\subsection{\correc{The variance parameter $c$ is known}}
\label{SecTestKnown}

After this contextualization, we are now ready to define a statistic that will be useful to achieve our objectives (recall that $\lambda_1=\pm 1$ means that the limit process has a unit root located at $\pm 1$). \correc{Suppose first that the variance parameter $c$ is known and consider, provided that $\lambda_1 \wh{v}_n(\alpha) < 1$, the statistic}
\begin{equation}
\label{EstAlpha}
\forall\, n > 1,\, \hsp \wh{\alpha}_n(\alpha) = \frac{\ln c - \ln\big( 1 - \lambda_1 \wh{v}_n(\alpha) \big)}{\ln n}
\end{equation}
where $\wh{v}_n(\alpha) = e_p\T\, \whvn(\alpha)$ is the first component of $\whvn(\alpha)$. \correc{Despite appearances, $\wh{\alpha}_n(\alpha)$ is \textit{not} an estimator of $\alpha$ because it actually depends on $\alpha$. It is a statistic whose behavior, as a function of $\alpha$, will help to formalize the test.}

\begin{thm}
\label{ThmNormAlpha}
Assume that \ref{HypCA}, \ref{HypUR} and \ref{HypSR} hold, and that $\dE[\vert \veps_1 \vert^{\, 2+\nu}] = \eta_{\nu} < +\infty$ for some $\nu > 0$. Then, for the `true' values $0 < \alpha < 1$ and $c > 0$,
\begin{equation}
(\ln n)\, \sqrt{n^{1-\alpha}} \, \big( \wh{\alpha}_n(\alpha) - \alpha) \cvgd \cN\big( 0,\, 2\, c^{-1} \pi_{11}^{-2} \big)
\end{equation}
where
\begin{equation}
\label{DefPi11}
    \pi_{11} = \frac{\lambda_1^{\, p-1}}{\prod_{j=2}^p (1 - \lambda_j)} \hsp (= 1 ~\text{if $p=1$}).
\end{equation}
\end{thm}
\begin{proof}
See Section \ref{SecProofThmAlpha}.
\end{proof}

Now choose some $\alpha/2 < \alpha_0 < 1$ (in particular, any $1/2 \leq \alpha_0 < 1$ is suitable), and consider the test statistic
\begin{equation}
\label{DefStatTest}
Z_n^{\, 2}(\alpha_0) = \frac{c\, \wh{\pi}_n^{\, 2}}{2}\, (\ln n)^2\, n^{1-\alpha_0}\, \big( \wh{\alpha}_n(\alpha_0) - \alpha_0 \big)^2 
\end{equation}
where
\begin{equation*}
\wh{\pi}_n = \frac{\lambda_1^{\, p-1}}{\prod_{i=2}^p (1 - \wh{\lambda}_{n,\, i})} \hsp (= 1 ~\text{if $p=1$})
\end{equation*}
and $\wh{\lambda}_{n,\, i}$ is the $i$-th eigenvalue of the estimated companion matrix $\wh{A}_n$ built by plugging $\whn$ into $A_n$, ordered as indicated in the hypotheses. Whenever $\wh{\alpha}_n(\alpha_0)$ is not defined for the reason mentioned above, we set $Z_n^{\, 2}(\alpha_0) = +\infty$ by convention (so that $\cH_0$ below is rejected). The following result is devoted to the asymptotic behavior of $Z_n^{\, 2}(\alpha_0)$ under the null $\cH_0 : ``\alpha = \alpha_0"$ as well as the alternative $\cH_1 : ``\alpha > \alpha_0"$.
\begin{thm}
\label{ThmTest}
Assume that \ref{HypCA}, \ref{HypUR} and \ref{HypSR} hold, and that $\dE[\vert \veps_1 \vert^{\, 2+\nu}] = \eta_{\nu} < +\infty$ for some $\nu > 0$. Then, under $\cH_0 : ``\alpha = \alpha_0"$, the test statistic \eqref{DefStatTest} satisfies 
\begin{equation}
\label{CvgZnH0}
Z_n^{\, 2}(\alpha_0) \cvgd \chi_1^2
\end{equation}
where $\chi_1^2$ has a chi-square distribution with one degree of freedom. If we additionally assume that $\alpha < 2 \alpha_0$, then under $\cH_1 : ``\alpha > \alpha_0"$,
\begin{equation}
\label{CvgZnH1}
\inf_{\alpha/2 \, <\, \alpha_0\, <\, \alpha}\, Z_n^{\, 2}(\alpha_0) \cvgp + \infty.
\end{equation}
\end{thm}
\begin{proof}
See Section \ref{SecProofThmTest}.
\end{proof}
From now on, we replace the assumption $\alpha < 2 \alpha_0$ by the stronger one $\alpha_0 \geq 1/2$ (this will be justified in the empirical section). As a corollary, for some $\alpha_0 \geq 1/2$, the rejection of $\cH_0 : ``\alpha = \alpha_0"$ in favor of $\cH_1 : ``\alpha > \alpha_0"$ will happen as soon as $Z_n^{\, 2}(\alpha_0) > z_{1-\epsilon}$ where $z_{1-\epsilon}$ is the quantile of order $1-\epsilon$ of the $\chi_1^2$ distribution. All the empirical study of Section \ref{SecEmp} is based on this rule \correc{(but note that instead of \eqref{DefStatTest}, we could as well use the first component of $\whvn(\alpha_0)$ in \eqref{OLS2} and reach the same conclusions, through the limiting distribution of Lemma \ref{LemCvgVarTheta}).} To evaluate the extent of instability, a natural choice may be, given $c$,
\begin{equation}
\label{TestAlphaMax}
\wh{\alpha}_c = \inf\, \{ \alpha_0 \in \cA\,:\, \text{$\cH_0$ is not rejected} \}
\end{equation}
where $\cA \subset\, [1/2, 1[$ is a grid of test values (by convention $\inf \emptyset = +\infty$ so that our conclusion is that the series is integrated if  $\cH_0$ is always rejected and $\alpha \geq 1/2$). It remains to generalize the procedure to the less common situation where the limit process has a negative unit root, that is $\lambda_1=-1$. Actually, Theorems \ref{ThmNormAlpha} and \ref{ThmTest} are still valid: $\lambda_1$ already appears in the first component of $\vartheta_n(\alpha)$ in \eqref{ModVarTheta}, \textit{i.e.} $\lambda_{n,\, 1} = \lambda_1 \rho_n$, and in $\wh{\alpha}_n(\alpha)$ given in \eqref{EstAlpha}. The reader is referred to Section \ref{SecNegRoot} for more technical details.

\subsection{\correc{The variance parameter $c$ is estimated}}
\label{SecTestUnknown}

\correc{As mentioned in the beginning of the section, the over-parametrization makes it impossible to dissociate $c$ from $\alpha$ without additional constraints. We suggest instead a cross-validation with the following protocol:
\begin{itemize}
\item For each $c_0 \in \cC$, compute $\wh{\alpha}_{c_0}$ based on the first $(n-n_{val})$ values of the series.
\item Choose $(\wh{c}_0,\, \wh{\alpha}_{\hat{c}_0})$ as the pair that optimizes some criterion stemming from the one-step predictions of the $n_{val}$ remaining values.
\end{itemize}
Clearly, the value of $\alpha$ selected by \eqref{TestAlphaMax} has to be biased downwards (due to overlapping confidence intervals for increasing values of $\alpha_0$). Thus, we must also expect $\wh{c}_0$ to be biased downwards because of the ratio $c/n^{\alpha}$. Paradoxically, despite unsatisfactory individual evaluations of $c$ and $\alpha$, the compensated result may prove interesting, as the empirical section suggests (see Figure \ref{FigSim6}). In terms of computational costs, proceeding like this is a great faster than testing all $(c, \alpha) \in \cC \times \cA$ since the cross-validation criterion is only computed once for each $c$. Overall, the procedure contains $O(1+\vert\cC\vert\vert\cA\vert + \vert\cC\vert n_{val})$ steps, where in this context a `step' requires at worst the inversion of a $p \times p$ matrix ($p \ll n$). Hence, this is not so large considering that $n_{val}$ is usually a small percentage of $n$. Especially as from a practical point of view, it does not seem necessary for $\cC$ to have a very fine mesh (unlike $\cA$).}

\section{Empirical Study}
\label{SecEmp}

This section is dedicated to the empirical performances of the test procedure described in Theorem \ref{ThmTest} together with the selection rule \eqref{TestAlphaMax}, and also to the cross-validation process presented in Section \ref{SecTestUnknown}.

\subsection{On simulated data (c is known)}
\label{SecEmpSimKnown}

Without loss of generality, let us consider $c=1$ to produce the data, and set $\rho_n = 1 -  n^{-\alpha} = \pm \lambda_{n,\, 1} $ with $1/2 \leq \alpha < 1$. For each simulation, $\lambda_{n,\, 2}, \ldots, \lambda_{n,\, p}$ are randomly chosen in the interval $[-\rho_n+0.1,\, \rho_n-0.1]$, and a process of size $n$ is generated with a standard Gaussian noise. The test of $\cH_0 : ``\alpha = \alpha_0"$ against $\cH_1 : ``\alpha > \alpha_0"$ is done \textit{via} Theorem \ref{ThmTest} for $\alpha_0 \in \cA = \{ 1/2 + k/50,\, k = 0, \ldots, 24  \}$. The rejection frequencies are evaluated on the basis of 5000 replications for each $\alpha_0$. The distribution of $Z_n^{\, 2}(\alpha_0)$ and the empirical power of the test with 5\% type I risk are represented on Figure \ref{FigSim1} ($\alpha=2/3$), Figure \ref{FigSim2} ($\alpha=3/4$) and Figure \ref{FigSim3} ($\alpha=4/5$), for $\lambda_1=1$, $n=1000$ and $p \in \{ 1, 2, 3, 4 \}$. To get an idea of the impact of smaller samples and a negative unit root, similar experiments are conducted on Figure \ref{FigSim4} (with $\lambda_1=1$, $\alpha=4/5$, $p=1$, $n=250$ on the top, and $\lambda_1=-1$, $\alpha=4/5$, $p=1$, $n=1000$ on the bottom). The light-gray items (at the right of the vertical line) correspond to $\alpha_0 > \alpha$, situations beyond the scope of our study that we decided anyway to incorporate in the simulations for informational purposes. Finally, the distribution of $\wh{\alpha}_{c\,=\,1}$ over 5000 replications is outlined for $\alpha \in \{ 2/3, 3/4, 4/5 \}$ with $n=250$ and $n=1000$, on Figure \ref{FigSim5}. For the latter simulations, we have adopted the rule \eqref{TestAlphaMax} with 5\% type I risk.

\begin{figure}[h]
\centering \includegraphics[scale=0.15]{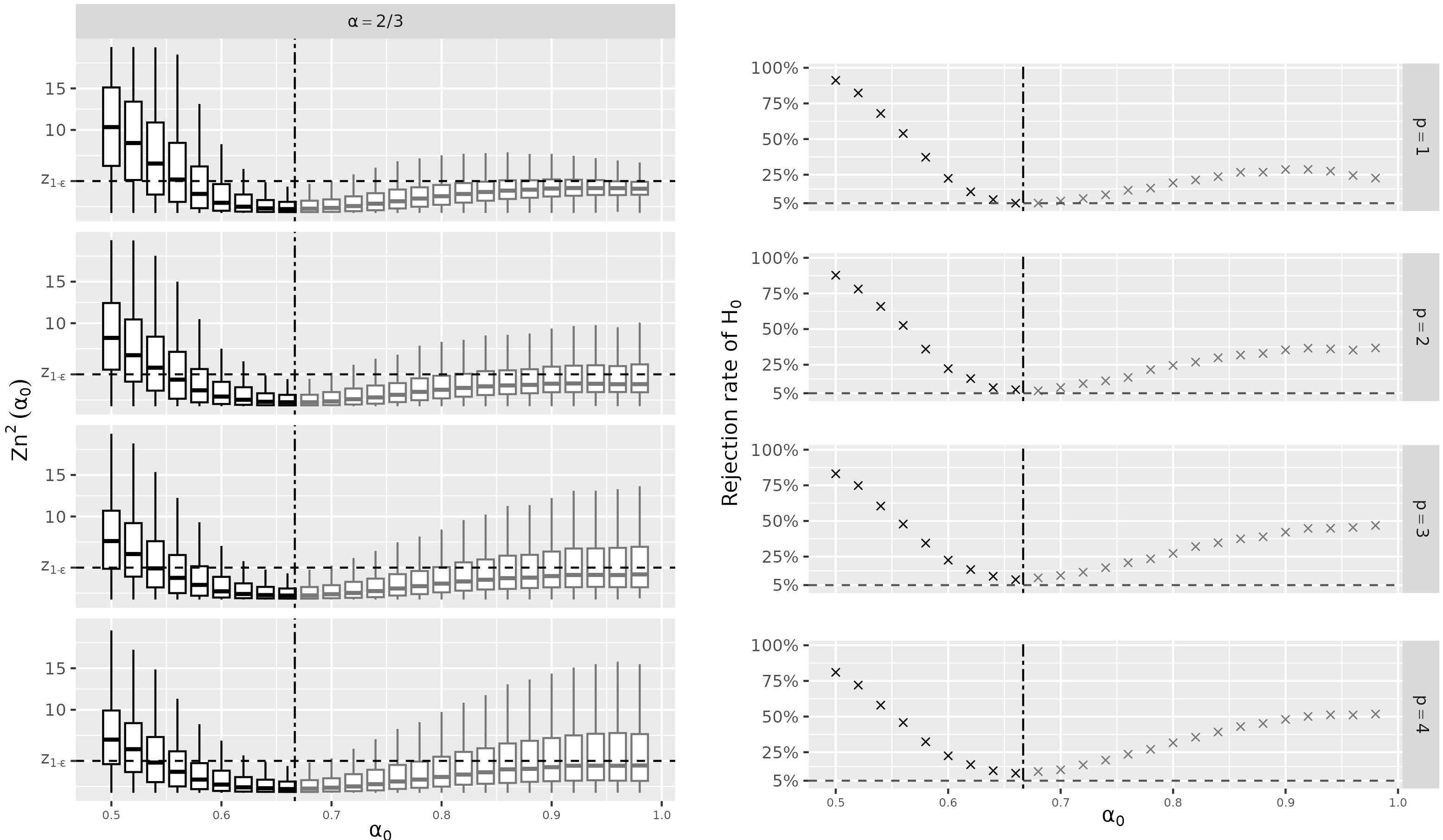}
\caption{Empirical distribution of $Z_n^{\, 2}(\alpha_0)$ for $n=1000$, $\lambda_1=1$, $\alpha=2/3$, $\alpha_0 \in \cA$ and $p \in \{ 1, 2, 3, 4 \}$ on the left. The  corresponding rejection frequencies of $\cH_0 : ``\alpha = \alpha_0"$ against $\cH_1 : ``\alpha > \alpha_0"$ (5\% level) are represented on the right.}
\label{FigSim1}
\end{figure}

\begin{figure}[h]
\centering \includegraphics[scale=0.15]{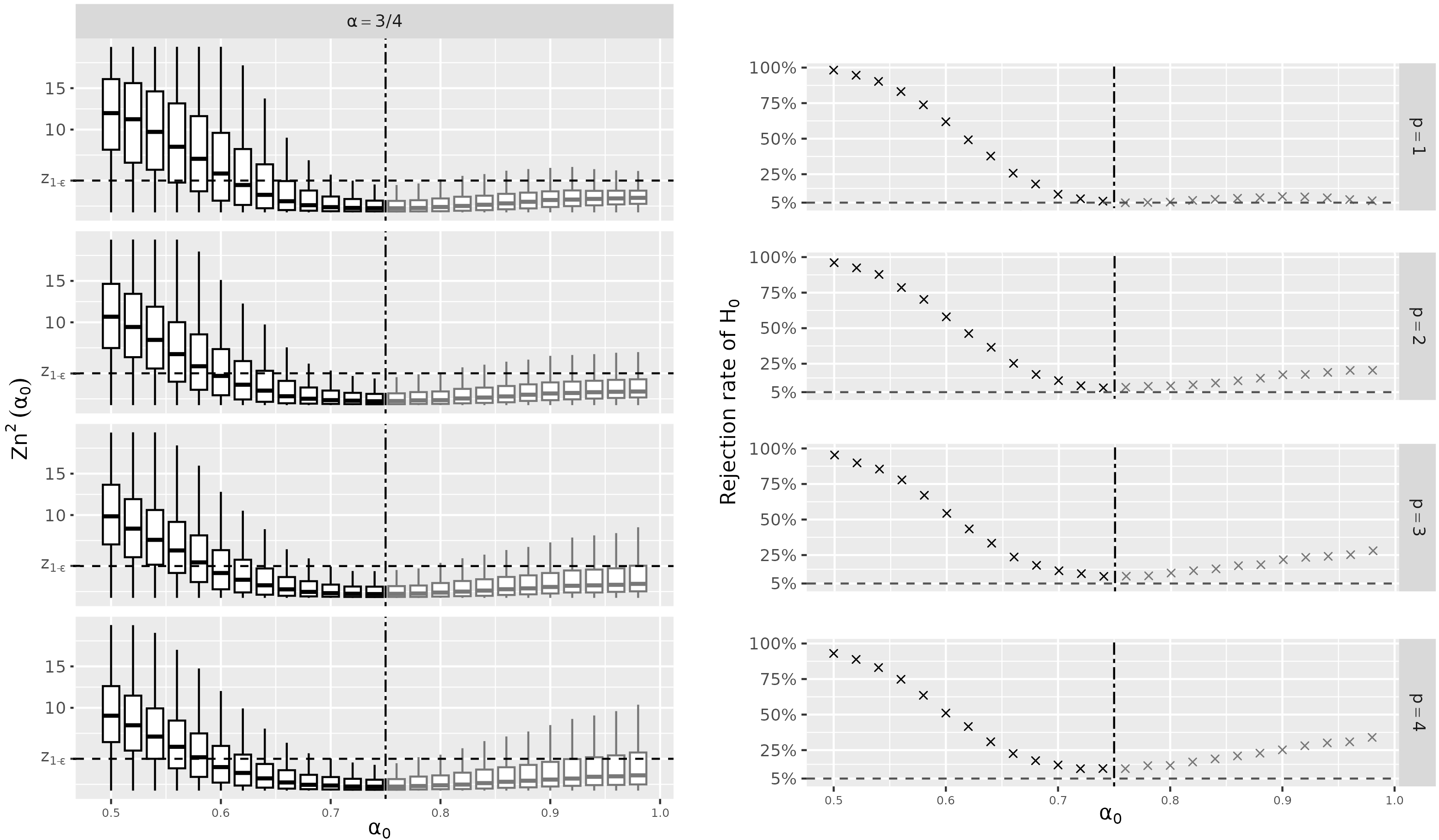}
\caption{Empirical distribution of $Z_n^{\, 2}(\alpha_0)$ for $n=1000$, $\lambda_1=1$, $\alpha=3/4$, $\alpha_0 \in \cA$ and $p \in \{ 1, 2, 3, 4 \}$ on the left. The  corresponding rejection frequencies of $\cH_0 : ``\alpha = \alpha_0"$ against $\cH_1 : ``\alpha > \alpha_0"$ (5\% level) are represented on the right.}
\label{FigSim2}
\end{figure}

\begin{figure}[h]
\centering \includegraphics[scale=0.15]{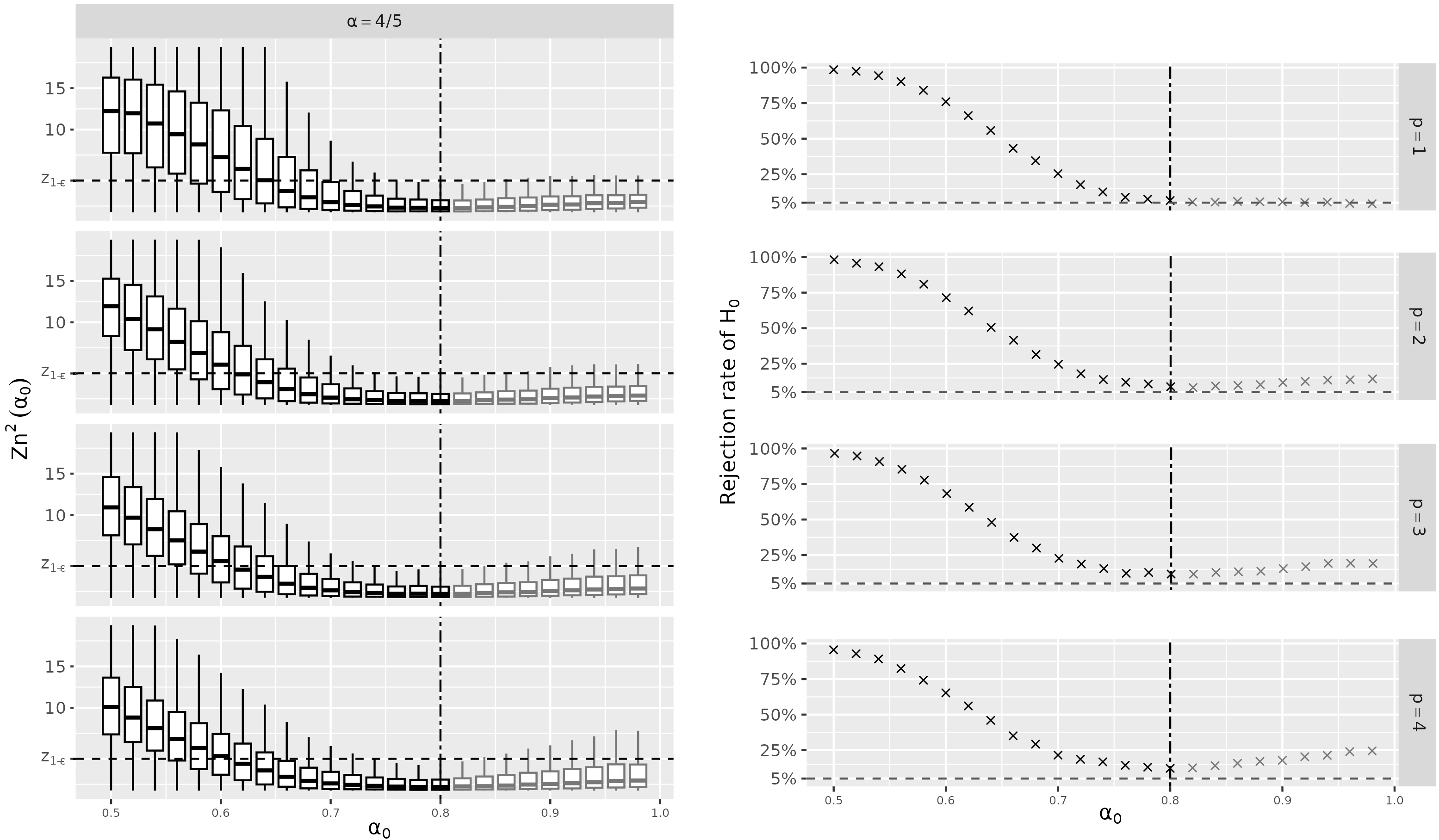}
\caption{Empirical distribution of $Z_n^{\, 2}(\alpha_0)$ for $n=1000$, $\lambda_1=1$, $\alpha=4/5$, $\alpha_0 \in \cA$ and $p \in \{ 1, 2, 3, 4 \}$ on the left. The  corresponding rejection frequencies of $\cH_0 : ``\alpha = \alpha_0"$ against $\cH_1 : ``\alpha > \alpha_0"$ (5\% level) are represented on the right.}
\label{FigSim3}
\end{figure}

\smallskip

From the simulations of Figures \ref{FigSim1}--\ref{FigSim3}, we observe that $p$ has a negligible influence on the asymptotic behaviors when $\lambda_1=1$. Under $\cH_0$, the rejection rate is $\approx 5\%$ as expected and the power of the test is quite satisfying under $\cH_1$ ($\geq 75\%$ for $\alpha_0 \leq \alpha-0.2$, $\approx 100\%$ for $\alpha_0 \leq \alpha-0.3$). The decrease in the rejection rates is rather similar for $\alpha_0$ growing from $1/2$ to $\alpha$, whether $\alpha=2/3$, $\alpha=3/4$ or $\alpha=4/5$. However, we can see that the trickier alternatives where $\alpha_0 \leq \alpha-0.1$ are underestimated ($\approx 25\%$): this is due to the fine mesh of $\cA$, it is difficult to discern substantial differences in the statistic between values of $\alpha_0$ that increases 0.02 by 0.02 and overall, the numerical results are encouraging.

\begin{figure}[h]
\centering \includegraphics[scale=0.15]{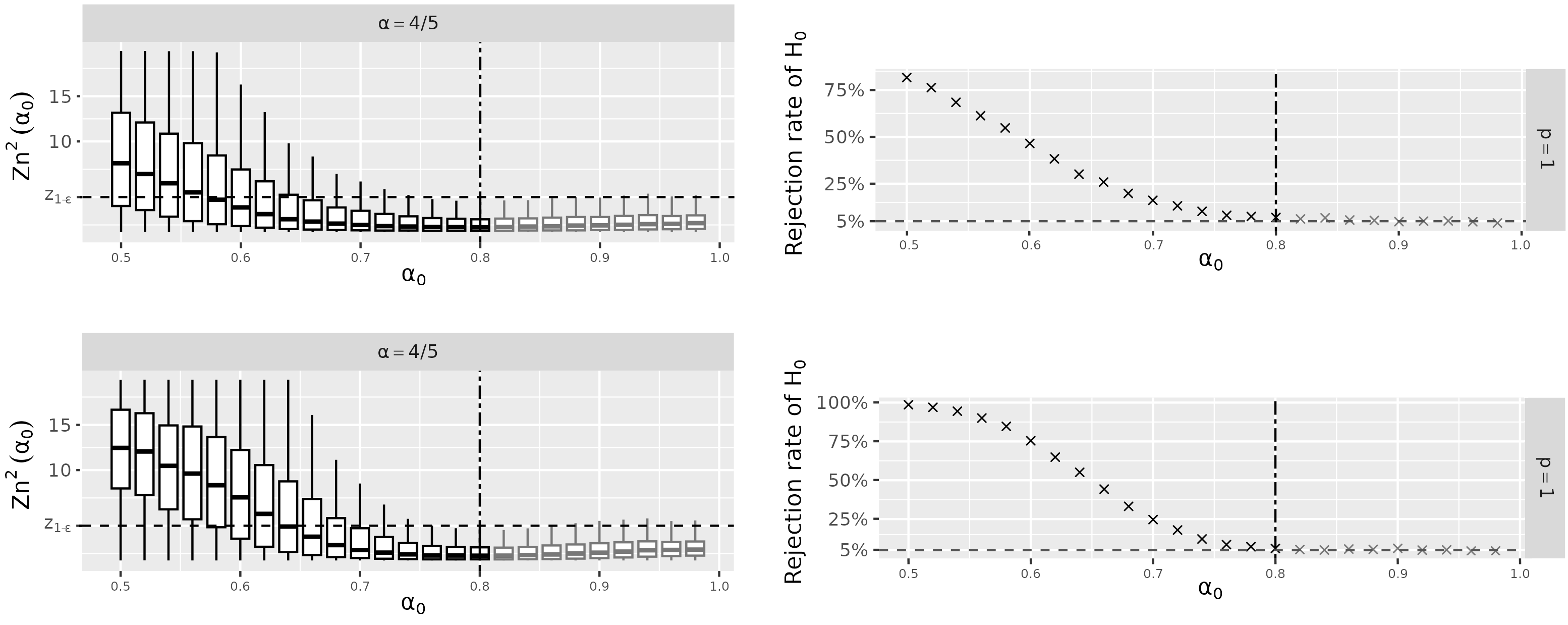}
\caption{Empirical distribution of $Z_n^{\, 2}(\alpha_0)$ for $n=250$, $\lambda_1=1$, $\alpha=4/5$, $\alpha_0 \in \cA$ and $p=1$ on the top left. The  corresponding rejection frequencies of $\cH_0 : ``\alpha = \alpha_0"$ against $\cH_1 : ``\alpha > \alpha_0"$ (5\% level) are represented on the right. Same experiments, but with $n=1000$, $\lambda_1=-1$, $\alpha=4/5$, $\alpha_0 \in \cA$ and $p=1$ on the bottom.}
\label{FigSim4}
\end{figure}

\smallskip

In terms of sample sizes, the results are obviously impacted when we set $n=250$ as can be seen on the upper part of Figure \ref{FigSim4} with $p=1$ and $\alpha=4/5$ ($\geq 50\%$ for $\alpha_0 \leq \alpha-0.2$), but the global behavior follows the same lines. The real time series of next subsection are even smaller ($n=80$ to $n=129$), that will mainly result in larger confidence intervals. Under the hypothesis that the limit process has a negative unit root, Figure \ref{FigSim4} (on the bottom) also highlights the same efficiency but we must precise here that, unlike $\lambda_1=1$, more experiments than those presented here showed a notable degradation with $p$: the decrease is faster and the statistic falls below $z_{1-\epsilon}$ too quickly. However, despite the lack of power observed for $p \geq 2$, this remains an innovative approach since we recall that the ADF, KPSS, and all usual tests for integration or trend-stationarity are unable to detect instability coming from negative unit roots.  

\begin{figure}[h]
\centering \includegraphics[scale=0.35]{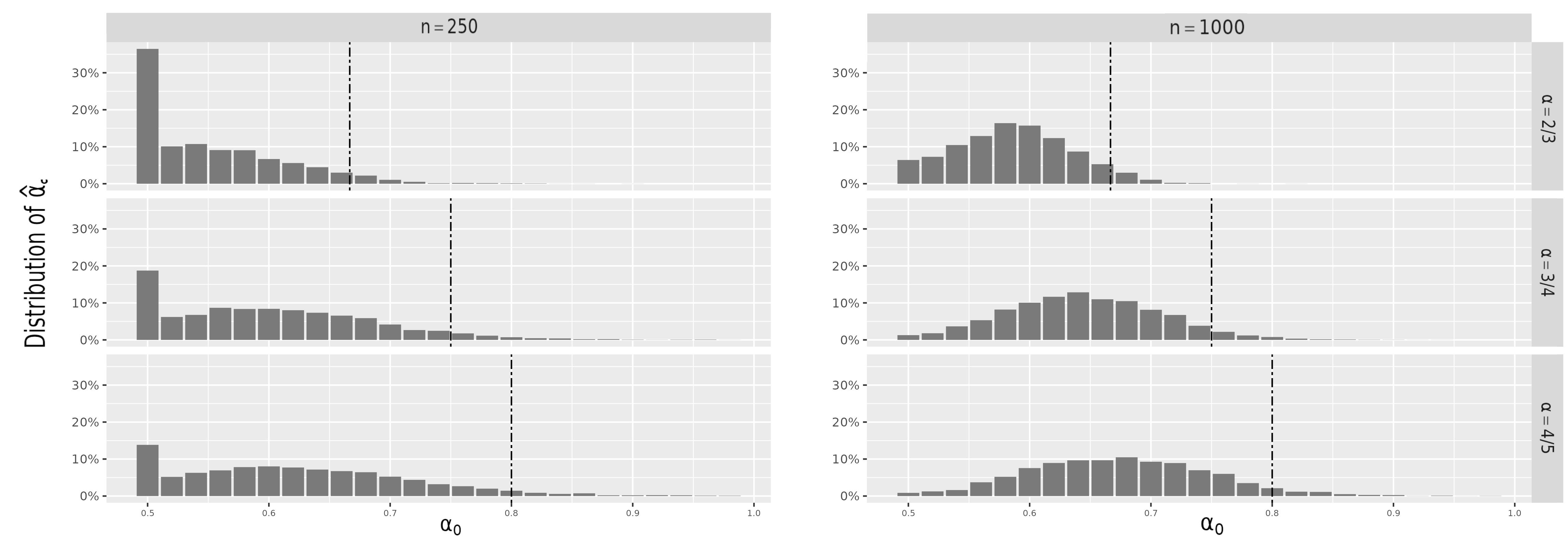}
\caption{Empirical distribution of $\wh{\alpha}_c$ (5\% level) with known value $c=1$, for $n=250$ (left) and $n=1000$ (right), $\lambda_1=1$, $\alpha \in \{ 2/3, 3/4, 4/5 \}$, $\alpha_0 \in \cA$ and $p=1$.}
\label{FigSim5}
\end{figure}

\smallskip

As expected, the selected values of $\alpha$ on Figure \ref{FigSim5} are underestimated, on average: for $n=1000$, we can roughly say that $\wh{\alpha}_{c\,=\,1}$ is close to $(\alpha-0.1)$, yet greater than $(\alpha-0.2)$ with an overwhelming probability. It is a direct consequence of the construction rule \eqref{TestAlphaMax} and overlapping confidence intervals (in particular, together with $\inf \cA = 1/2$, that explains the initial peak on smaller samples). \correc{So, the spectral radius must be underestimated as well. Fortunately the next section will show that, by evaluating $c$ together with $\alpha$ instead of arbitrarily setting $c=1$, the estimation is clearly improved.} Of course a fully unbiased strategy may be an interesting trail for future studies (even if the procedure is asymptotically unbiased), see the conclusion for some comments on the topic.

\subsection{\correc{On simulated data (c is estimated)}}
\label{SecEmpSimUnknown}

\correc{The estimation of $c$ according to the cross-validation described in Section \ref{SecTestUnknown} actually proves beneficial. Indeed, let us consider three groups of experiments generated as in the previous section, with $n=1000$, $\lambda_1=1$, $p=2$ and $(\alpha,\, c)$ uniformly chosen out of $\cA \times \cC = \{ 1/2 + k/50,\, k = 0, \ldots, 24 \} \times \{ 1/2 + k/5,\, k = 0, \ldots, 22  \}$. Figure \ref{FigSim6} shows the scatterplot of $\wh{c}/n^{\wh{\alpha}}$ as a function of $c/n^{\alpha}$ together with a linear fit (with 95\% confidence interval) for each of them. In the first one, we basically set $c=1$ and estimate $\alpha$ using \eqref{TestAlphaMax}. This corresponds to the decreasing linear fit on the plot and reveals a significant deterioration, hence the inadequacy of making $\wh{c}=1$ independent of the data: the bias of $\wh{\alpha}_{\hat{c}\,=\,1}$ may strongly affect the results. The second group of experiments consists of giving $\wh{c}=c$ the true value before computing $\wh{\alpha}_{\hat{c}\,=\,c}$ by the same strategy. The linear fit is clearly better, but the parallelism with the bisector of the coordinate axes unveils a systematic bias. This is a predictable consequence of overlapping confidence intervals of \eqref{TestAlphaMax} already mentioned in the previous section: $\wh{\alpha}_{\hat{c}}$ is underestimated so that $c/n^{\wh{\alpha}_{\hat{c}}}$ overestimates the true value and the points are generally above the identity function. Finally, in the third group of experiments we estimate the pair $(c, \alpha)$ by the cross-validation of Section \ref{SecTestUnknown} with $c_0 \in \cC$ defined above, $n_{val} = n/10$ and the MSE criterion. One can see that except for the small values of $c/n^{\alpha}$ (either $c$ is small or $\alpha$ is large) corresponding to a spectral radius close to one and where all methods perform quite similarly, only the last one is able to counterbalance the bias of $\wh{\alpha}_{\hat{c}}$ and results in a more reliable estimation. Incidentally, the t-tests for equality of slopes in the groups confirm the visual intuition through a clear rejection. The latter approach is now going to be used on a real dataset.}

\begin{figure}[h]
\centering \includegraphics[scale=0.33]{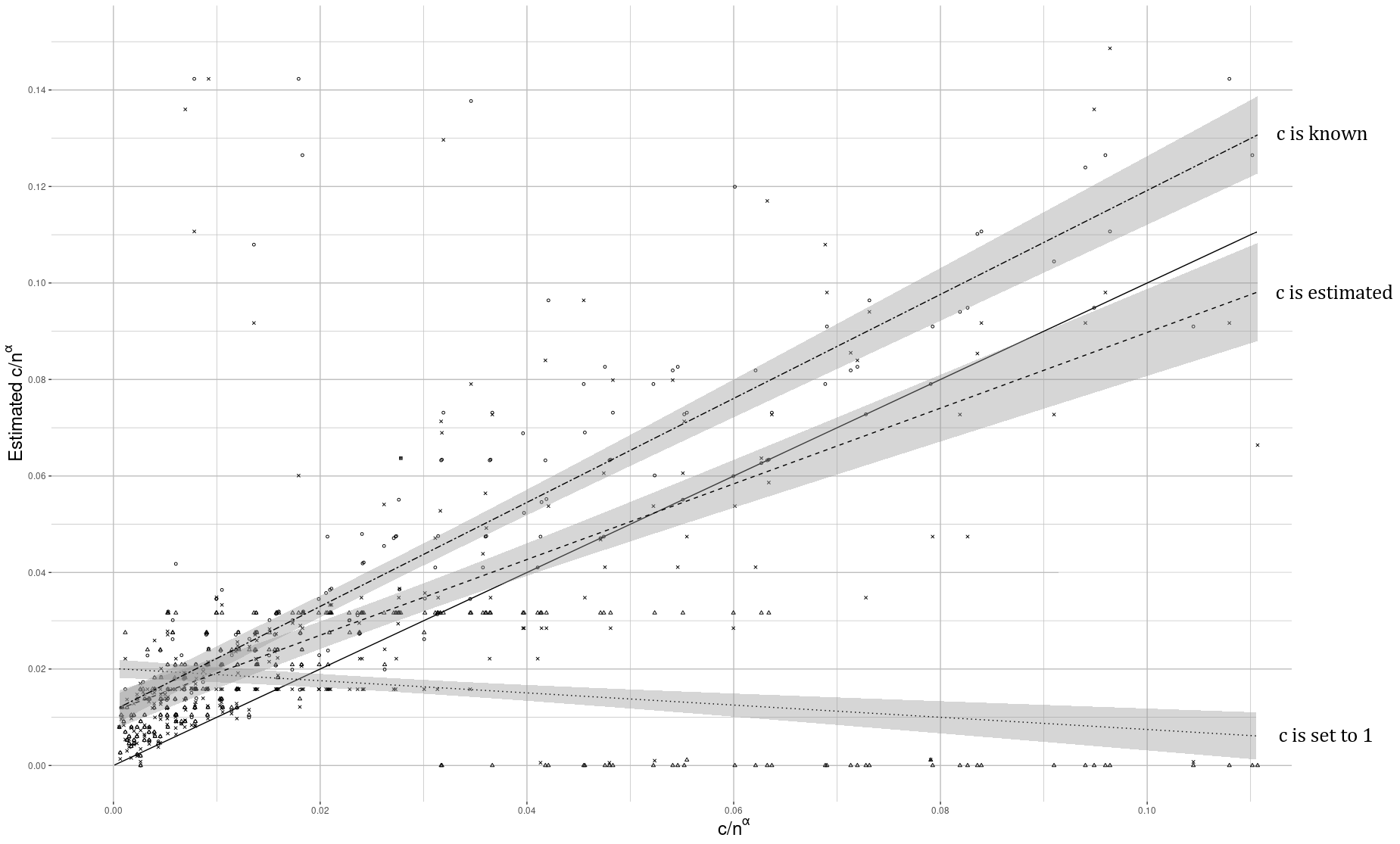}
\caption{\correc{Scatterplot of $\wh{c}/n^{\wh{\alpha}}$ as a function of $c/n^{\alpha}$ in three different cases: $c$ is set to 1 (triangles), $c$ is known (circles) and $c$ is estimated by cross-validation (crosses). Experiments are conducted according to the given protocol to simulate the data and estimate $\alpha$. A linear fit (with 95\% confidence interval) shows the mean evolution in each group and the solid line is the identity function}.}
\label{FigSim6}
\end{figure}

\subsection{On a real dataset}
\label{SecEmpReal}

The extended dataset \texttt{NelPlo} is available as open-source \url{https://www.rdocumentation.org/packages/tseries/versions/0.9-9/topics/NelPlo} and contains 14 macroeconomic series ending in 1988. According to \citet{NelsonPlosser82}, those data are known to be either integrated (in the sense of the ADF test) or trend-stationary with spectral radius close to 1, so let us approach them from the nearly-unstable point of view, with positive unit root ($\lambda_1 = 1$). It is important to note that the data of \texttt{NelPlo} are longer than the original ones so that our conclusions may be slightly different. The statistical protocol itself is different: $p$ is chosen by looking at the partial autocorrelations of the differentiated series and no trend is estimated. \correc{Then, the pair $(\wh{c},\,\wh{\alpha}_{\hat{c}})$ is selected through the cross-validation described in Section \ref{SecTestUnknown}, with $n_{val} = n/10$ and the MSE criterion. An (approximate) 95\%-confidence interval of $\wh{\alpha}_{\hat{c}}$ is also derived from Theorem \ref{ThmNormAlpha} with $c=\wh{c}$. Because $n$ is usually small, in any case much smaller than in the simulation study, we have restrained the domain of $c$ and only considered a search grid $\cC \subset [0.8, 1.2]$, while as before $\cA \subset [0.5, 1[$.} See Table \ref{TabNelPloUnknownC} below for the numerical outcomes.

\begin{table}[h]
\begin{center}
\begin{tabular}{|c|c|c|c|c|c|}
\hline
Series & $n$ & $p$ & $\wh{c}$ & $\wh{\alpha}_{\hat{c}}$ & $\wh{\rho}_n = 1 - \wh{c}/n^{\wh{\alpha}_{\hat{c}}}$ \\
\hline
Velocity & 120 & 1 & 0.80 & $\in [0.50,\, 0.66]$ & $\in [0.93,\, 0.97]$ \\
\hline
Industrial production & 129 & 6 & 0.80 & $\in [0.50,\, 0.72]$ & $\in [0.93,\, 0.98]$ \\
\hline
Nominal GNP & 80 & 2 & 0.80 & $\in [0.50,\, 0.57]$ & $\in [0.91,\, 0.93]$ \\
\hline
Consumer prices & 129 & 4 & 1.08 & $\in [0.73,\, 0.98]$ & $\in [0.97,\, 0.99]$ \\
\hline
Employment & 99 & 5 & 0.80 & $\in [0.50,\, 0.59]$ & $\in [0.92,\, 0.95]$ \\
\hline
Interest rate & 89 & 4 & 0 & $+\infty$ & 1 \\
\hline
Wages & 89 & 2 & 0.80 & $\in [0.50,\, 0.56]$ & $\in [0.92,\, 0.94]$ \\
\hline
GNP deflator & 100 & 6 & 1.20 & $\in [0.50,\, 0.53]$ & $\in [0.88,\, 0.90]$ \\
\hline
Money stock & 100 & 3 & 0.80 & $\in [0.50,\, 0.53]$ & $\in [0.92,\, 0.93]$ \\
\hline
Real GNP & 80 & 4 & 1.10 & $\in [0.50,\, 0.74]$ & $\in [0.88,\, 0.96]$ \\
\hline
Common stock prices & 118 & 6 & 0.82 & $\in [0.50,\, 0.98]$ & $\in [0.92,\, 0.99]$ \\
\hline
Real per capita GNP & 80 & 2 & 0.80 & $\in [0.50,\, 0.62]$ & $\in [0.91,\, 0.95]$ \\
\hline
Real wages & 89 & 2 & 1.04 & $\in [0.50,\, 0.79]$ & $\in [0.89,\, 0.97]$ \\
\hline
Unemployment rate & 99 & 3 & 1.08 & $\in [0.50,\, 0.65]$ & $\in [0.89,\, 0.95]$ \\
\hline
\end{tabular}
\end{center}
\caption{\texttt{NelPlo} dataset from the nearly-unstable point of view. \correc{For each series, $(\wh{c},\,\wh{\alpha}_{\hat{c}})$ coming from the cross-validation of Section \ref{SecTestUnknown} is given together with an (approximate) 95\% confidence interval for $\wh{\alpha}_{\hat{c}}$.} The estimated spectral radius is also precised.}
\label{TabNelPloUnknownC}
\end{table}

\smallskip

Although the statistical context is different, our results are rather consistent with Table 5 of \citet{NelsonPlosser82}. According to these authors, all series except `Unemployment rate' are integrated. \correc{In fact it seems that among them, only `Interest rate' has a unit root whereas the nearly-unstable approach may be suitable for the others. It should be noted that the only stationary series of the authors (radius $\approx 0.71$) corresponds to one of the lowest spectral radius given by our procedure. Similarly, the most integrated series of the authors is `Interest rate' (radius $\approx 1.03$), that is the only one being considered as integrated in our study. There are obviously points of divergence due to different protocols, to the fact that we did not try to detrend the series beforehand, and so on, but it is interesting to note that we share the same general ideas on these data. For illustrative purposes, Figures \ref{FigNelPlo1}-\ref{FigNelPlo4} contain some curves from the dataset and associated forecasts.}

\begin{figure}[h]
\centering \includegraphics[scale=0.45]{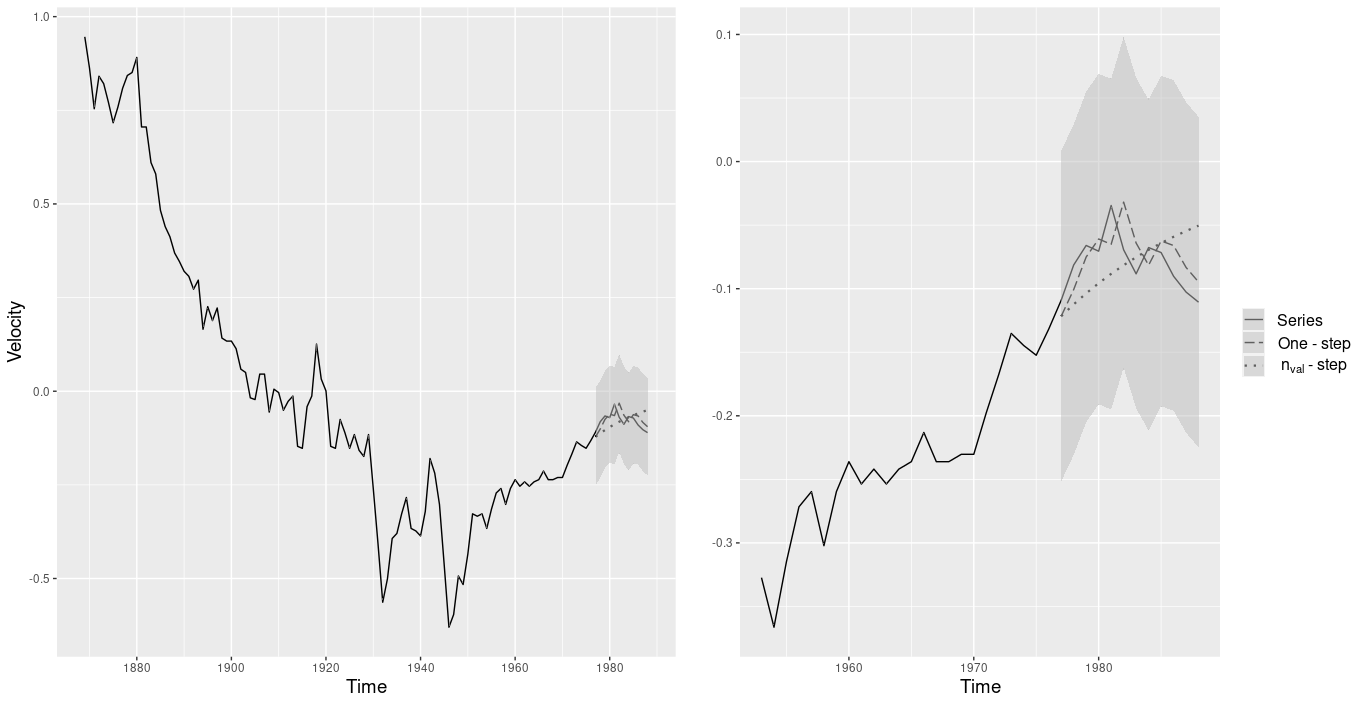}
\caption{\correc{Series `Velocity' and associated predictions for the last $n_{val} = 12$ values (one-step and $n_{val}$-step forecasts). The Gaussian 95\% confidence interval for one-step predictions is added. The whole series is represented on the left whereas on the right a zoom focuses on the last decades.}}
\label{FigNelPlo1}
\end{figure}

\begin{figure}[h]
\centering \includegraphics[scale=0.45]{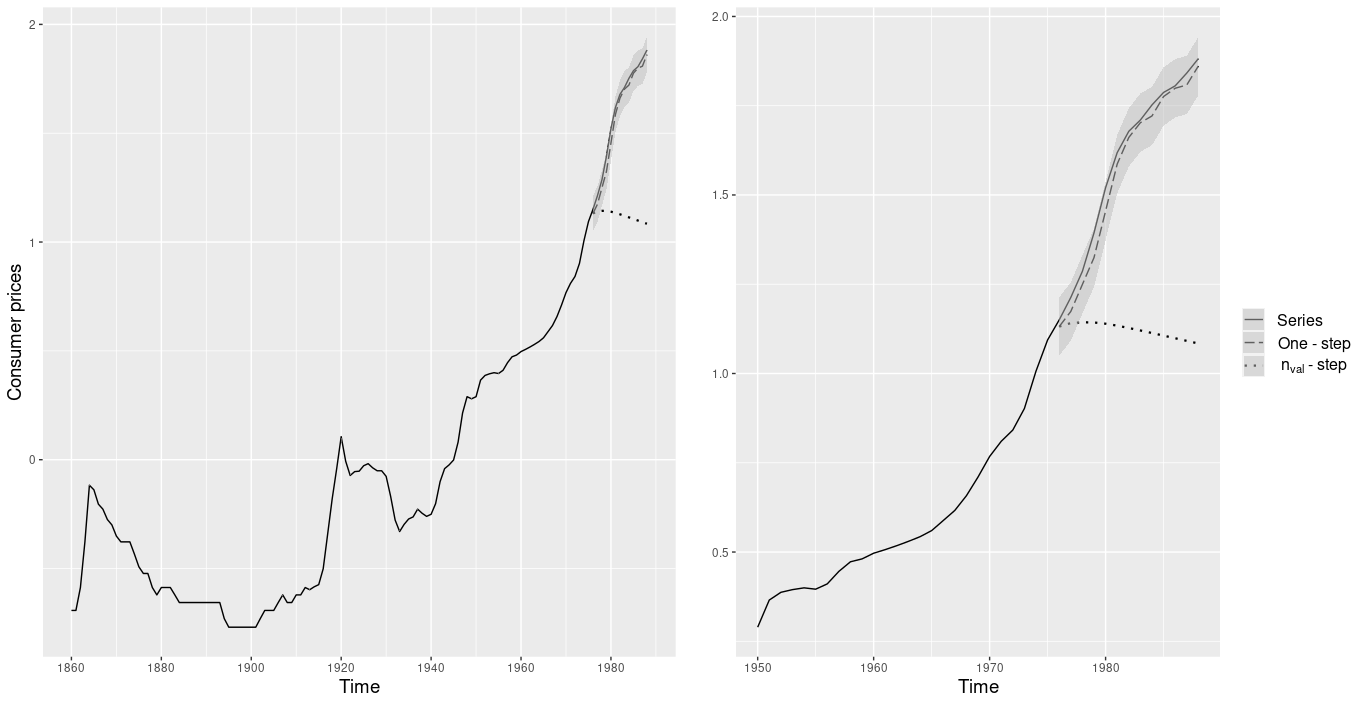}
\caption{\correc{Same graphs as in Figure \ref{FigNelPlo1} with the series `Consumer prices' and $n_{val} = 13$ values.}}
\label{FigNelPlo2}
\end{figure}

\begin{figure}[h]
\centering \includegraphics[scale=0.45]{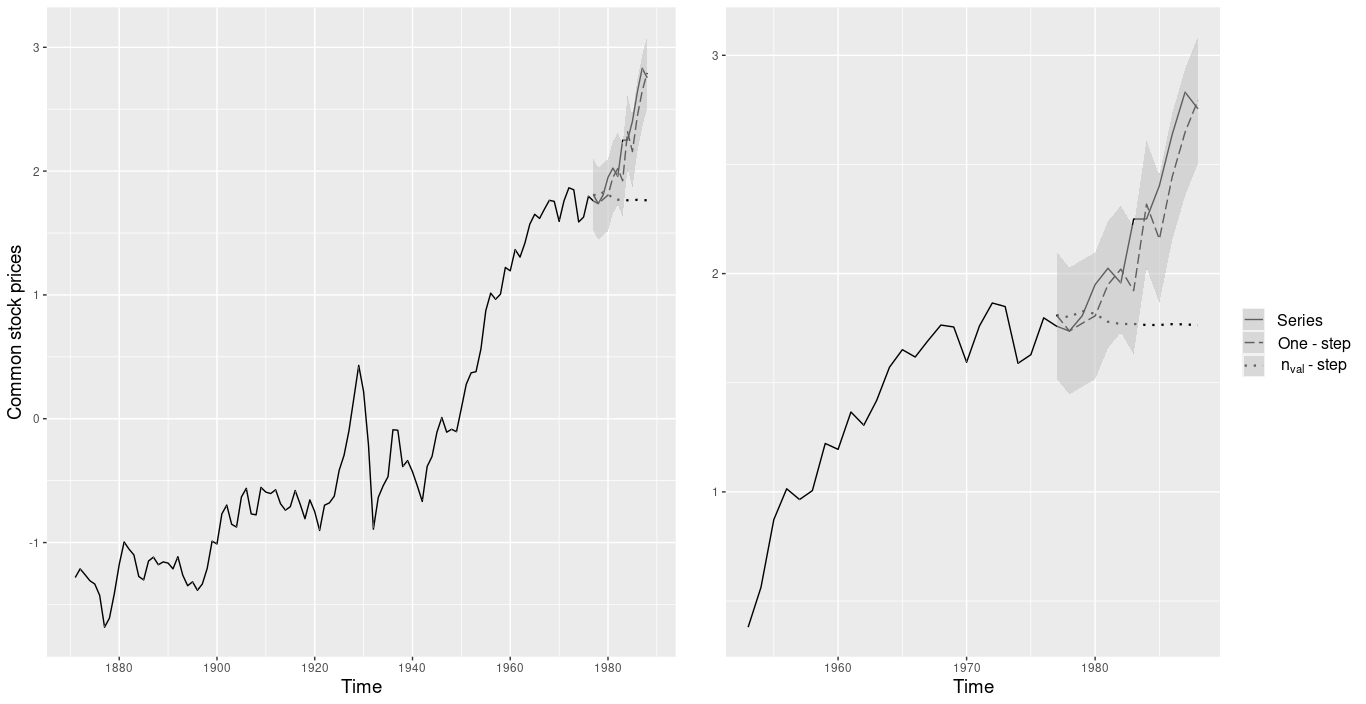}
\caption{\correc{Same graphs as in Figure \ref{FigNelPlo1} with the series `Common stock prices' and $n_{val} = 12$ values.}}
\label{FigNelPlo3}
\end{figure}

\begin{figure}[h]
\centering \includegraphics[scale=0.45]{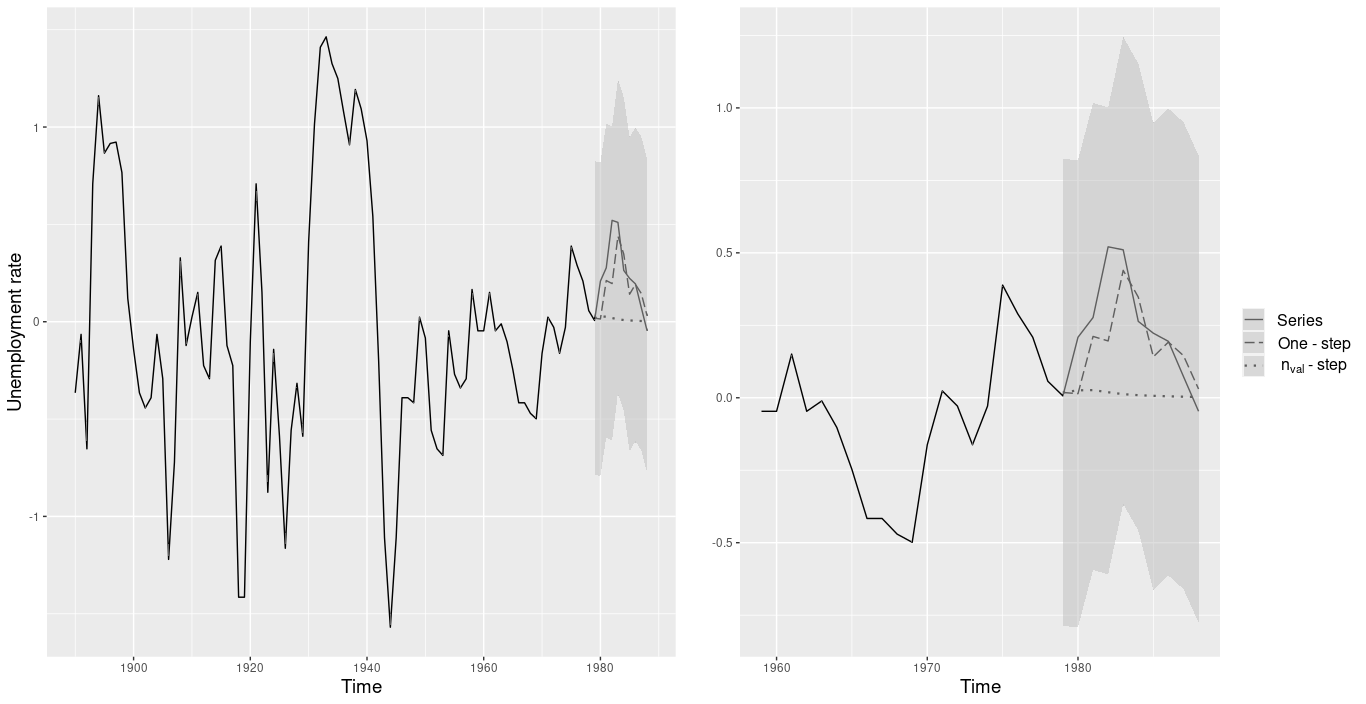}
\caption{\correc{Same graphs as in Figure \ref{FigNelPlo1} with the series `Unemployment rate' and $n_{val} = 10$ values.}}
\label{FigNelPlo4}
\end{figure}

\section{\correc{Conclusion and Perspectives}}
\label{SecConclu}

\correc{Our motivation was to bring a fresh look to the unit root issues in time series analysis, through the concept of near-instability. On a dubious series in terms of stationarity, the `extent of instability' may now be subject to a testing procedure, and many opportunities of improvements should follow from this innovative approach.} For example, it would be interesting to unveil the limit behavior of the test statistic with no constraints on $\alpha$. The light-gray boxplots of Figures \ref{FigSim1}--\ref{FigSim3} suggest that $\cH_1 : ``\alpha < \alpha_0"$ does not lead to a divergence but to some kind of stabilization away from zero (in distribution ?). But a spectral radius close to unity together with $\alpha < 1/2$ must correspond to high values of $n$, questioning the practical interest (a series with moderate size is more likely stable in that context). On the other hand, when $\alpha \geq 1/2$, a bilateral test would undoubtedly refine the selection of $\wh{\alpha}_{c}$ and contribute to eliminate the observed bias, so this is an important issue. \correc{Also, the choice of $c$ presented here is based on empirical foundations, which makes it difficult to obtain theoretical grounds related to consistency or normality. The authors are pretty convinced that a joint estimation of $(c, \alpha)$ through a penalized Gaussian likelihood may be a solution to improve the estimation, bypass over-parametrization through well-chosen constraints and overall, get guarantees.} Authorizing $p = p_n$ to grow to infinity could lead to the formalization of nearly-unstable ARMA$(p, q)$ processes under an hypothesis of invertibility at fixed $n$ and some restrictions on $p_n$ -- like the strategy of ADF. The results of this paper shall then probably be generalized without much difficulty to that wider class of models. \correc{Especially as, on a broader scale, any model with a critical parameter could be responsive to this kind of approach: the features of the reproduction law for counting and Galton-Watson processes (see \citet{IspanyEtAl03}, \citet{IspanyEtAl05}), the mean value of the random coefficients in RCAR processes, the fractional order of integration in FARIMA processes, the Hurst exponent of long-term memory time series, and so on. For $p$-dimensional processes, the first step should be to express the stability/instability criterion through the spectral radius of a companion matrix. For example in \citet{BarczyEtAl11}, an INAR$(p)$ process with $\rho(A)<1$ is stable whereas it is unstable for $\rho(A)=1$ (see $A$ in (2.2), Prop. 2.6 and the definition that follows). Hence, through probably arduous demonstrations, it seems that our approach might be adaptable: let $\rho(A_n)<1$ with $\rho(A_n) \rightarrow 1$, and consider near-unstability in such a count time series to generalize from first-order to $p$-order. Finally, a natural question arising here concerns explosive nearly-unstable processes, that is when the spectral radius tends to one from the  \textit{outer} neighborhood of the unit circle. Although less interesting from a practical point of view, this seems to be a challenging trail for future studies.}

\section{Technical Developments}
\label{SecProofs}

To deal with martingale terms, we need to consider the triangular array of filtrations
\begin{equation}
\label{Filt}
\forall\, n \geq 1,\, \forall\, 1 \leq k \leq n, \hsp \cF_{n,\, k} = \sigma(\Phi_{n,\, 0},\, \veps_1,\, \ldots,\, \veps_k)
\end{equation}
obviously satisfying $\cF_{n,\, 0} \subseteq \cF_{n,\, 1} \subseteq \ldots \subseteq \cF_{n,\, n}$ for each $n \geq 1$.

\subsection{Known results}
\label{SecKnown}

The two lemmas given in this section may be found with far more details in \citet{BadreauProia2023}.

\begin{lem}
\label{LemDiag}
Assume that \ref{HypCA} holds. Then, there exists $n_0 \geq 1$ such that, for all $n > n_0$, $A_n$ is diagonalizable in the form $A_n = P_n\, D_n\, P_n^{-1}$ with $D_n = \diag(\lambda_{n,\, 1},\, \ldots,\, \lambda_{n,\, p})$ containing ordered distinct eigenvalues $1\, >\, \vert \lambda_{n,\, 1} \vert\, \geq\, \ldots\, \geq\, \vert \lambda_{n,\, p} \vert\, >\, 0$. In addition, $\Vert P_n \Vert$ and $\Vert P_n^{-1} \Vert$ are bounded and there exists an invertible matrix $P$ such that $P_n \rightarrow P$ and $P_n^{-1} \rightarrow P^{-1}$.
\end{lem}

Note that because its spectrum is a subset of the one of $A_n$, the companion matrix $\bar{A}_{n}$ associated with $(V_{n,\,k})$ is also diagonalizable for all $n > n_0$ and the lemma still holds when replacing $A_n$ by $\bar{A}_n$ and obviously $p$ by $p-1$. This will be a key argument to justify \eqref{CvgTn22} in the next section.

\begin{lem}
\label{LemCvgSn}
Assume that \ref{HypCA}, \ref{HypUR} and \ref{HypSR} hold and that $\dE[\veps_1^{\, 2}] = \sigma^2 < +\infty$. Then, as $n$ tends to infinity,
\begin{equation*}
(1 - \rho_n(\alpha))\, \frac{S_{n,\, n}}{n} \cvgp \Gamma_p \eqdef \sigma^2 \ell \left( \lambda_1^{i+j} \right)_{1\, \leq\, i, j\, \leq\, p} \hsp \text{with} \hsp \ell = \frac{\pi_{11}^2}{2} > 0,
\end{equation*}
where $S_{n,\, n}$ is given in \eqref{DefSnTn} and $\pi_{i1}$ ($i = 1, \ldots, p$) are the (non-zero) entries of the first column of $P^{-1}$. Moreover,
\begin{equation*}
C_n^{-\frac{1}{2}}(\alpha)\, P_n^{-1}\, \frac{S_{n,\, n}}{n}\, (P_n\T)^{-1} C_n^{-\frac{1}{2}}(\alpha) \cvgp \sigma^2\, \begin{pmatrix}
\frac{\pi_{11}^2}{2} & 0 & \hdots & 0 \\
0 & \frac{\pi_{21}^2}{1 - \lambda_2^2} & \hdots & \frac{\pi_{21} \pi_{p1}}{1 - \lambda_2 \lambda_p} \\
\vdots & \vdots & \ddots & \vdots \\
0 & \frac{\pi_{p1} \pi_{21}}{1 - \lambda_p \lambda_2} & \hdots & \frac{\pi_{p1}^2}{1 - \lambda_p^2}
\end{pmatrix} \eqdef H_p
\end{equation*}
where $C_n(\alpha)$ is defined in \eqref{RateCn}. In addition, $H_p$ is invertible (and positive definite if $A$ has only real eigenvalues).
\end{lem}

\subsection{Technical lemmas}
\label{SecTechnical}

\begin{lem}
\label{LemCvgTn}
Assume that \ref{HypCA}, \ref{HypUR} and \ref{HypSR} hold and that $\dE[\veps_1^{\, 2}] = \sigma^2 < +\infty$. Then, as $n$ tends to infinity, for any $\alpha/2 < \alpha_0 \leq \alpha$,
\begin{equation*}
C_n^{-\frac{1}{2}}(\alpha)\, \frac{T_{n,\, n}(\alpha_0)}{n}\, C_n^{-\frac{1}{2}}(\alpha) \cvgp \Upsilon_p \eqdef \sigma^2\, \diag\left( \ell,\, \Sigma_{p-1} \right)
\end{equation*}
where $T_{n,\, n}(\alpha)$ is defined in \eqref{DefSnTn}, $C_n(\alpha)$ in \eqref{RateCn}, and where $\ell$ and $\Sigma_{p-1}$ are defined in Lemma \ref{LemCvgSn} and \eqref{DefSigma}, respectively. In addition, $\Upsilon_p$ is positive definite.
\end{lem}
\begin{proof}
Note that $S_{n,\, n}$ and $T_{n,\, n}(\alpha)$ share the same top-left element (which actually does not depend on $\alpha$), so we directly get
\begin{equation}
\label{CvgTn11}
\frac{1 - \rho_n(\alpha)}{n} \sum_{k=0}^n X^2_{n,\, k} = \frac{c}{n^{1+\alpha}} \sum_{k=0}^n X^2_{n,\, k} \cvgp \sigma^2 \ell.
\end{equation}
Moreover, coming back to the reasoning of \citet{BadreauProia2023}, especially Lem. 13--15 and Prop. 2, and adapting it to the present case (see the remark after Lemma \ref{LemDiag}),
\begin{equation}
\label{CvgTn22}
\frac{1}{n} \sum_{k=0}^n \Lambda_{n,\, k}(\alpha)\, \Lambda_{n,\, k}\T(\alpha) \cvgp \sigma^2\, \Sigma_{p-1} \hsp \text{where} \hsp \Lambda_{n,\, k}\T(\alpha) = \begin{pmatrix} V_{n,\, k}(\alpha) & \cdots & V_{n,\, k-p+2}(\alpha) \end{pmatrix}.
\end{equation}
Indeed, the equivalence
\begin{equation*}
\frac{1}{n} \sum_{k=0}^n \Lambda_{n,\, k}(\alpha)\, \Lambda_{n,\, k}\T(\alpha) ~ \overset{p}{\sim} ~ \left( \sum_{k\, \geq\, 0} \bar{A}_n^{\, k}\, K_{p-1}\, (\bar{A}_n\T)^{\, k} \right) \frac{1}{n} \sum_{k=1}^n \veps_k^{\, 2},
\end{equation*}
the law of large numbers together with the fact that $\rho(\bar{A}_n) < 1$ and $\rho(\bar{A}) < 1$ lead to \eqref{CvgTn22}. It remains to treat the (symmetrical) non-diagonal blocks. First,
\begin{eqnarray}
\label{DefNonDiagBlocks}
\sum_{k=0}^n \Lambda^{\! (p)}_{n,\, k}(\alpha)\, X_{n,\, k} & = & \Lambda^{\! (p)}_{n,\, 0}(\alpha) X_{n,\, 0} + \left[ \sum_{k=1}^n \left( \Phi_{n,\, k} - \rho_n(\alpha)\, \Phi_{n,\, k-1} \right) \Phi\T_{n,\, k} \right] e_p \\
 & = & \left[ S_{n,\, n} \left( I_p - \rho_n(\alpha) A_n\T \right) - \rho_n(\alpha) M_{n} + r_n \right] e_p \nonumber
\end{eqnarray}
where $\Lambda^{\! (p)}_{n,\, k}$ is the $p$-dimensional version of $\Lambda_{n,\, k}$ (with last element $V_{n,\, k-p+1}(\alpha)$),
\begin{equation*}
M_{n} = \sum_{k=1}^n \Phi_{n,\, k-1}\, E\T_k
\end{equation*}
and $r_n$ is made of isolated terms. It is already proven (Lem. 14--15) that $\Vert M_{n} \Vert$ and $\Vert r_n \Vert$ are $o_p(n)$. As an intermediate result, we can see that for $n > n_0$,
\begin{eqnarray}
\label{DefDiagDelta}
C_n^{\frac{1}{2}}(\alpha)\, P_n\T \left( I_p - \rho_n(\alpha) A_n\T \right) & = & C_n^{\frac{1}{2}}(\alpha) \left( I_p - \rho_n(\alpha) D_n \right) P_n\T \\
 & \cvgp & \diag\left( 0, (1-\lambda_i)_{2\, \leq\, i\, \leq\, p} \right) P\T \eqdef \Delta_p\, P\T. \nonumber
\end{eqnarray}
Hence,
\begin{eqnarray}
\label{CvgTn21}
\frac{\sqrt{1 - \rho_n(\alpha)}}{n}\, S_{n,\, n} \left( I_p - \rho_n(\alpha) A_n\T \right) & = & P_n\, \frac{C_n^{\frac{1}{2}}(\alpha)\sqrt{c}}{n^{\frac{\alpha}{2}}} \left[ C_n^{-\frac{1}{2}}(\alpha)\, P_n^{-1}\, \frac{S_{n,\, n}}{n}\, (P_n\T)^{-1} C_n^{-\frac{1}{2}}(\alpha) \right] \nonumber \\
 & & \hsp \hsp \hsp \times \left[ C_n^{\frac{1}{2}}(\alpha)\, P_n\T \left( I_p - \rho_n(\alpha) A_n\T \right) \right] \nonumber \\
 & \cvgp & P\, K_p\, H_p\, \Delta_p\, P\T
\end{eqnarray}
where the block diagonal matrix $H_p$ is given in Lemma \ref{LemCvgSn} and $\Delta$ is defined above. But $K_p$, $H_p$ and $\Delta_p$ have particular structures consisting of zeros in key locations and implying the following pattern,
\begin{equation*}
P\, K_p\, H_p\, \Delta_p = \begin{pmatrix}
* & 0 & \hdots & 0 \\
\vdots & \vdots & & \vdots \\
\vdots & \vdots & & \vdots \\
* & 0 & \hdots & 0
\end{pmatrix} \begin{pmatrix}
0 & \hdots & \hdots & 0 \\
\vdots & * & \hdots & * \\
\vdots & \vdots & & \vdots \\
0 & * & \hdots & *
\end{pmatrix} = 0.
\end{equation*}
The combination of \eqref{CvgTn11}, \eqref{CvgTn22}, \eqref{CvgTn21} and the remarks that follow lead to the convergence
\begin{equation}
\label{eq_CvgTn}
C_n^{-\frac{1}{2}}(\alpha)\, \frac{T_{n,\, n}(\alpha)}{n}\, C_n^{-\frac{1}{2}}(\alpha) \cvgp \Upsilon_p
\end{equation}
that is, when $\alpha_0$ takes the `true' value $\alpha$. We now wish to generalize the result to $T_{n,\, n}(\alpha_0)$ for any $\alpha/2 < \alpha_0 \leq \alpha$. Regarding this, note that
\begin{equation}
\label{DefPsiPoint}
\Psi_{n,\, k}(\alpha_0) = \Psi_{n,\, k}(\alpha) + c \left( n^{-\alpha_0} - n^{-\alpha} \right) \dot{\Phi}_{n,\, k}
\end{equation}
where
\begin{equation*}
\forall\, n \geq 1,\, \forall\, 1 \leq k \leq n, \hsp \dot{\Phi}\T_{n,\, k} = \begin{pmatrix} 0 & X_{n,\, k-1} & \cdots & X_{n,\, k-p+1} \end{pmatrix}.
\end{equation*}
 Thus, we have the decomposition
\begin{equation}
\label{DecompTn}
T_{n,\, n}(\alpha_0) = T_{n,\, n}(\alpha) + R_n
\end{equation}
with
\begin{eqnarray*}
R_n & = & c \left( n^{-\alpha_0} - n^{-\alpha} \right) \left[ \sum_{k=0}^n \Psi_{n,\, k}(\alpha)\, \dot{\Phi}_{n,\, k}\T + \sum_{k=0}^n \dot{\Phi}_{n,\, k} \Psi_{n,\, k}\T(\alpha) \right] \\
 & & \hsp \hsp \hsp +~ c^2 \left( n^{-\alpha_0} - n^{-\alpha} \right)^2 \sum_{k=0}^n \dot{\Phi}_{n,\, k}\, \dot{\Phi}_{n,\, k}\T ~ \eqdef ~ R_{n,\, 1} + R_{n,\, 1}\T + R_{n,\, 2}.
\end{eqnarray*}
The direct computation shows that $\Psi_{n,\, k}(\alpha)\, \dot{\Phi}_{n,\, k}\T$ contains a column of zeros, a line made of $X_{n,\, k}\, X_{n,\, k-i}$ ($1 \leq i \leq p-1$), and a $(p-1) \times (p-1)$ block of terms of the form $V_{n,\, k-j}(\alpha)\, X_{n,\, k-i}$ ($0 \leq j \leq p-2$, $1 \leq i \leq p-1$). For $2 \leq i \leq p-1$, according to the first expression in \eqref{ModHiera},
\begin{align*}
\rho_{n}^{2i}(\alpha) \, X_{n,\,k-i} &= \rho_{n}^{2i-1}(\alpha) \left(X_{n,\,k-i+1} - V_{n,\,k-i+1}(\alpha)\right) \\
&= \rho_{n}^{i+1}(\alpha) \, X_{n,\,k-1} - \sum_{j=1}^{i-1} \rho_{n}^{2i-j}(\alpha) \, V_{n,\,k-i+j}(\alpha)
\end{align*}
which entails that, for any $0 \leq u \leq p-2$,
\begin{eqnarray}
\label{eq_XV}
\rho_{n}^{2i}(\alpha) \sum_{k=1}^n V_{n,\, k-u}(\alpha) \, X_{n,\,k-i}  & = & \rho_{n}^{i+1}(\alpha) \sum_{k=1}^n V_{n,\, k-u}(\alpha) \, X_{n,\,k-1} \nonumber \\
 & & \hsp \hsp \hsp -~ \sum_{j=1}^{i-1} \rho_{n}^{2i-j}(\alpha) \sum_{k=1}^n V_{n,\, k-u}(\alpha) \, V_{n,\,k-i+j}(\alpha).
\end{eqnarray}
The first part of this equation is $o_p(n^{1+\alpha/2})$ and the second part is a finite sum of $O_p(n)$, see for instance \eqref{eq_CvgTn}. Thus, \textit{via} Lemma \ref{LemCvgSn},
\begin{equation}
\label{CvgR1}
\left\Vert C_n^{-\frac{1}{2}}(\alpha)\, \frac{R_{n,\, 1}}{n}\, C_n^{-\frac{1}{2}}(\alpha) \right\Vert = O_p\left( \frac{ n^{-\alpha_0} - n^{-\alpha} }{n^{1 + \frac{\alpha}{2}}}\, n^{1 + \alpha} \right) = o_p(1)
\end{equation}
as soon as $\alpha_0 > \alpha/2$. In the same way, $R_{n,\,2}$ is closely related to the covariance matrix $S_{n,\,n}$ but with zeros on the first line and column. So,
\begin{equation}
\label{CvgR2}
\left\Vert C_n^{-\frac{1}{2}}(\alpha)\, \frac{R_{n,\,2}}{n}\, C_n^{-\frac{1}{2}}(\alpha) \right\Vert = O_p\left( \frac{ (n^{-\alpha_0} - n^{-\alpha})^2 }{n}\, n^{1 + \alpha} \right) = o_p(1)
\end{equation}
provided that $\alpha_0 > \alpha/2$. Using \eqref{eq_CvgTn}, \eqref{DecompTn}, \eqref{CvgR1} and \eqref{CvgR2}, the convergence is now established. Since $\ell > 0$ (Lemma \ref{LemCvgSn}), the positive-definiteness of $\Upsilon_p$ relies on the invertibility of $\Sigma_{p-1}$, which can be found \textit{e.g.} in \citet{BrockwellDavis91} in the context of a stable AR$(p-1)$ process with $\sigma^2 > 0$.
\end{proof}

\begin{lem}
\label{LemCvgVarTheta}
Assume that \ref{HypCA}, \ref{HypUR} and \ref{HypSR} hold and that $\dE[\vert \veps_1 \vert^{\, 2+\nu}] = \eta_{\nu} < +\infty$. Then, as $n$ tends to infinity, for the `true' value of $\alpha$,
\begin{equation}
\label{NormVarTheta}
\sqrt{n} \, C_n^{\frac{1}{2}}(\alpha) \left( \whvn(\alpha) - \vartheta_n(\alpha) \right) \cvgd \cN\big( 0,\, \sigma^2\, \Upsilon_p^{-1} \big)
\end{equation}
where the limit covariance $\Upsilon_p$ is given in Lemma \ref{LemCvgTn}. Moreover,
\begin{equation}
\inf_{\alpha/2 \, <\, \alpha_0\, <\, \alpha}\, \left\Vert \sqrt{n} \, C_n^{\frac{1}{2}}(\alpha_0) \left( \whvn(\alpha_0) - \vartheta_n(\alpha_0) \right) \right\Vert \cvgp +\infty.
\end{equation}
\end{lem}
\begin{proof}
Developing $X_{n,\, k} = \vartheta_n\T(\alpha) \Psi_{n,\, k-1}(\alpha) + \veps_k$ in \eqref{OLS2}, the estimation error satisfies 
\begin{equation}
\label{DecomptEstErr}
\sqrt{n} \, C_n^{\frac{1}{2}}(\alpha) \left( \whvn(\alpha) - \vartheta_n(\alpha) \right) = \sqrt{n} \, C_n^{\frac{1}{2}}(\alpha)\, T_{n,\, n-1}^{\, -1}(\alpha) \sum_{k=1}^{n} \Psi_{n,\, k-1}(\alpha)\, \veps_k.
\end{equation}
The sequence $(\Psi_{n,\, k-1}(\alpha)\, \veps_k)_{n,\, k}$ is a martingale-difference array w.r.t. $\cF_{n,\, k}$ defined in \eqref{Filt} so, for all $u \in \dR^p\backslash\{0\}$,
\begin{equation}
\label{Mart}
M_{n,\, n}^{(u)} = u\T\, \frac{C_n^{-\frac{1}{2}}(\alpha)}{\sqrt{n}} \sum_{k=1}^n \Psi_{n,\, k-1}(\alpha)\, \veps_k
\end{equation}
is a martingale array having predictable quadratic variation
\begin{eqnarray}
\label{CvgCroch}
\langle M^{(u)} \rangle_{n,\, n} & = & \sigma^2\, u\T C_n^{-\frac{1}{2}}(\alpha)\, \frac{T_{n,\, n-1}(\alpha)}{n}\, C_n^{-\frac{1}{2}}(\alpha)\, u \nonumber \\
 & \cvgp & \sigma^2\, u\T \Upsilon_p\, u > 0
\end{eqnarray}
by the convergence in Lemma \ref{LemCvgTn} (with $\alpha_0 = \alpha$) and the positive definiteness of $\Upsilon_p$. It is now necessary to establish the Lindeberg's condition in order to apply the central limit theorem for arrays of martingales (see \textit{e.g.} Thm. 1 of Sec. 8 of \citet{Pollard84}), namely that
\begin{equation}
\label{Lindeberg}
\forall\, \epsilon > 0, \hsp \sum_{k=1}^n \dE\left[ (\Delta M_{n,\, k}^{(u)} )^2\, \ind_{\{ \vert \Delta M_{n,\, k}^{(u)} \vert\, >\, \epsilon \}}\, \vert\, \cF_{n,\, k-1} \right] \cvgp 0
\end{equation}
where $\Delta M_{n,\, k}^{(u)} = M_{n,\, k}^{(u)} - M_{n,\, k-1}^{(u)}$. But for any $1 \leq k \leq n$,
\begin{equation}
\label{LindSup}
\dE\big[ (\Delta M_{n,\, k}^{(u)} )^2\, \ind_{ \{ \vert \Delta M_{n,\, k}^{(u)} \vert\, >\, \epsilon \} }\, \vert\, \cF_{n,\, k-1} \big] ~=~ \varphi^{(u)}_{n,\, k-1}(\alpha)\, \frac{\xi_{n,\, k}}{n} ~\leq~ \varphi^{(u)}_{n,\, k-1}(\alpha)\, \frac{\xi^*_n}{n}
\end{equation}
where $\varphi^{(u)}_{n,\, k-1}(\alpha) = u\T C_n^{-1/2}(\alpha)\, \Psi_{n,\, k-1}(\alpha)\, \Psi\T_{n,\, k-1}(\alpha)\, C_n^{-1/2}(\alpha)\, u$,
\begin{equation*}
\xi_{n,\, k} = \dE\big[ \veps_k^{\, 2}\, \ind_{ \{ \vert \Delta M_{n,\, k}^{(u)} \vert\, >\, \epsilon \} }\, \vert\, \cF_{n,\, k-1} \big] \hsp \text{and} \hsp \xi^*_n = \sup_{1\, \leq\, k\, \leq\, n} \xi_{n,\, k}.
\end{equation*}
Applying H\"older's and Markov's inequalities,
\begin{eqnarray*}
\forall\, 1 \leq k \leq n, \hsp \xi_{n,\, k}  & \leq & \eta_{\nu}^{\frac{2}{2+\nu}}\, \dP\big( (\Delta M_{n,\, k}^{(u)})^2 > \epsilon^2\, \vert\, \cF_{n,\, k-1} \big)^{\frac{\nu}{2+\nu}} \\
 & \leq & \eta_{\nu}^{\frac{2}{2+\nu}} \left( \frac{\sigma^2\, \varphi^{(u)}_{n,\, k-1}(\alpha)}{\epsilon^2\, n} \right)^{\! \frac{\nu}{2+\nu}} \cvgp 0
\end{eqnarray*}
from Lemma \ref{LemCvgTn} (with $\alpha_0 = \alpha$), where $\eta_{\nu}$ is the moment of order $2+\nu$ of the white noise $(\veps_k)$. Thus $\xi^*_n = o_p(1)$, see \textit{e.g.} Lem. 1.3.20 of \citet{Duflo97}. Together with \eqref{LindSup} and again Lemma \ref{LemCvgTn}, \eqref{Lindeberg} follows. Consequently, the martingale defined in \eqref{Mart} satisfies
\begin{equation*}
M_{n,\, n}^{(u)} \cvgd \cN(0,\, \sigma^2\, u\T \Upsilon_p\, u)
\end{equation*}
and because that holds for any $u \neq 0$, by the Cram\'er-Wold device,
\begin{equation}
\label{CvgMart}
\frac{C_n^{-\frac{1}{2}}(\alpha)}{\sqrt{n}} \sum_{k=1}^n \Psi_{n,\, k-1}(\alpha)\, \veps_k \cvgd \cN\big( 0,\, \sigma^2\, \Upsilon_p ).
\end{equation}
Combined with \eqref{CvgCroch}, Slutsky's lemma and \eqref{DecomptEstErr},
\begin{equation*}
\sqrt{n} \, C_n^{\frac{1}{2}}(\alpha) \left( \whvn(\alpha) - \vartheta_n(\alpha) \right) \cvgd \cN\big( 0,\, \sigma^2\, \Upsilon_p^{-1} \big).
\end{equation*}
That concludes the proof of the first part of the lemma. For the other part, note that using \eqref{ModVarTheta}, \eqref{DefPsiPoint} and \eqref{OLS2}, we have the decomposition
\begin{align}
\label{DecompVartheta}
\whvn(\alpha_0) - \vartheta_n(\alpha_0) &= \left[\vartheta_n(\alpha) - \vartheta_n(\alpha_0) \right] + T_{n,\, n-1}^{\, -1}(\alpha_0) \sum_{k=1}^{n} \Psi_{n,\, k-1}(\alpha_0) \, \veps_k \\
& \hsp \hsp \hsp +~ c \left(n^{-\alpha} - n^{-\alpha_0} \right) T_{n,\, n-1}^{\, -1}(\alpha_0) \sum_{k=1}^{n} \Psi_{n,\, k-1}(\alpha_0)\, \dot{\Phi}_{n,\, k-1}\T \, \vartheta_n(\alpha) \nonumber \\
& \eqdef \tilde{R}_{n,\,1} + \tilde{R}_{n,\,2} + \tilde{R}_{n,\,3}. \nonumber
\end{align}
The reasoning we have just developed also applies to $\tilde{R}_{n,\, 2}$ \textit{via} the same lines and tools, and as a result we obtain the convergence
\begin{equation}
\label{DecompCvgR2}
\sqrt{n} \, C_n^{\frac{1}{2}}(\alpha) \, \tilde{R}_{n,\,2} \cvgd \cN\big( 0,\,\sigma^2 \, \Upsilon_p^{-1} \big).
\end{equation}
That leads to the first intermediate step,
\begin{equation}
\label{CvgRn2}
\left\vert e_p\T \sqrt{n} \, C_n^{\frac{1}{2}}(\alpha_0) \, \tilde{R}_{n,\,2} \right\vert ~=~ \left\vert \sqrt{c \, n^{\alpha_0 - \alpha}} \, e_p\T \, \sqrt{n} \, C_n^{\frac{1}{2}}(\alpha) \, \tilde{R}_{n,\,2} \right\vert ~=~ o_p(1)
\end{equation}
as soon as $\alpha_0 < \alpha$. We can consider the rest of the decomposition similarly, that is
\begin{equation}
\label{DecompRn13}
\left\vert e_p\T \sqrt{n} \, C_n^{\frac{1}{2}}(\alpha_0) \left(\tilde{R}_{n,\,1} + \tilde{R}_{n,\,3}\right) \right\vert = \left\vert \, \sqrt{c \, n^{1-\alpha_0}} \left(1 - n^{\alpha_0 - \alpha}\right) \right\vert \vert 1 - u_n \vert
\end{equation}
where the first factor is clearly divergent when $\alpha_0 < \alpha$ and
\begin{equation*}
    u_n \eqdef e_p\T \, T_{n,\, n-1}^{\, -1}(\alpha_0) \left[ \sum_{k=1}^{n} \Psi_{n,\, k-1}(\alpha_0)\, \dot{\Phi}_{n,\, k-1}\T \right] \vartheta_n(\alpha).
\end{equation*}
As in the proof of Lemma \ref{LemCvgTn}, it is easy to see that $\Psi_{n,\, k-1}(\alpha_0)\, \dot{\Phi}_{n,\, k-1}\T$ is a $p \times p$ matrix made of a first column of zeros, terms of the form $X_{n,\, k-1}\, X_{n,\, k-j}$ on the first line and $V_{n,\, k-i}(\alpha_0)\, X_{n,\, k-j}$ in the bottom-right block $(1 \leq i \leq p-1,\, 2 \leq j \leq p)$, broken down as
\begin{eqnarray*}
\rho_{n}^{2j}(\alpha) \sum_{k=1}^n V_{n,\, k-i}(\alpha_0) \, X_{n,\,k-j}  & = & \rho_{n}^{j+1}(\alpha) \sum_{k=1}^n V_{n,\, k-i}(\alpha_0) \, X_{n,\,k-1} \\
 & & \hsp \hsp \hsp -~ \sum_{s=1}^{j-1} \rho_{n}^{2j-s}(\alpha) \sum_{k=1}^n V_{n,\, k-i}(\alpha_0) \, V_{n,\,k-j+s}(\alpha)
\end{eqnarray*}
as it is done in \eqref{eq_XV}. As soon as $\alpha/2 < \alpha_0 < \alpha$, Lemma \ref{LemCvgTn} entails that the first part of this equation is $o_p(n^{1+\alpha/2})$ and, using Cauchy–Schwarz inequality,
\begin{equation*}
\left\vert\sum_{k=1}^n V_{n,\, k-i}(\alpha_0) \, V_{n,\,k-j+l}(\alpha) \right\vert 
~\leq~ \sqrt{\sum_{k=1}^n V^{\, 2}_{n,\, k-i}(\alpha_0)} \, \sqrt{\sum_{k=1}^n V^{\, 2}_{n,\,k-j+l}(\alpha)}
\end{equation*}
where both sums are $O_p(n)$. Finally, it is straightforward to obtain the simplification
\begin{equation}
\label{CvgUn}
    u_n \cvgp e_p\T \, \Upsilon_p^{-1} \, \sigma^2 \, \ell \, \begin{pmatrix}
    0 & 1 & \hdots & 1 \\
    \vdots & 0 & \hdots & 0 \\
    \vdots & \vdots & & \vdots \\
    0 & 0 & \hdots & 0 \\
    \end{pmatrix} \, \begin{pmatrix}
    1 \\
    \vdots \\
    \beta_i \\
    \vdots \\
    \end{pmatrix} ~=~ \sum_{i=1}^{p-1} \beta_i \eqdef s_{\beta}
\end{equation}
where $\ell$ and $\Upsilon$ are defined in Lemmas \ref{LemCvgSn} and \ref{LemCvgTn}, respectively. Note that $1 - s_{\beta}$ is obviously non-zero since by stability, the polynomial $\Theta_\beta(z) = 1 - \beta_1 z - \ldots - \beta_{p-1} z^{p-1}$ has no zero on the unit circle. Going back to \eqref{CvgRn2} and \eqref{DecompRn13}, we then have, for any $\alpha/2 < \alpha_0 < \alpha$,
\begin{equation}
\label{DivVarthetAlpha0}
    \left\vert e_p\T \sqrt{n} \, C_n^{\frac{1}{2}}(\alpha_0) \left( \whvn(\alpha_0) - \vartheta_n(\alpha_0) \right) \right\vert \cvgp +\infty
\end{equation}
which concludes the proof.
\end{proof}

\subsection{Proof of Proposition \ref{PropCvgNormOLS3}.}
\label{SecProofCorOLS3}
First, note that
\begin{eqnarray*}
\sqrt{n}\, \big( \wtn - \theta_n \big) & = & \sqrt{n}\, \big( \wh{J}_n - J_n \big)\, \big( \wh{\beta}_n - \beta_n \big) + \sqrt{n}\, \big( \wh{J}_n - J_n \big)\, \beta_n \\
 & & \hsp \hsp \hsp +~ \sqrt{n}\, J_n\, \big( \wh{\beta}_n - \beta_n \big) + \sqrt{n}\, \big( \wh{v}_n - \lambda_{n,\, 1} \big)\, e_p.
\end{eqnarray*}
In fact, the proof rests entirely on Lemma \ref{LemCvgVarTheta}. The first line of the asymptotic normality gives $\vert \wh{v}_n - \lambda_{n,\, 1} \vert = O_p(n^{-(1+\alpha)/2}) = o_p(n^{-1/2})$ and obviously $\Vert \wh{J}_n - J_n \Vert = o_p(n^{-1/2})$. Now from the remaining part of \eqref{NormVarTheta},
\begin{equation*}
\sqrt{n}\, \big( \wh{\beta}_n - \beta_n \big) \cvgd \cN\big( 0,\, \Sigma_{p-1}^{-1} \big)
\end{equation*}
where the (invertible) covariance $\Sigma_{p-1}$ is given in \eqref{DefSigma}. So $\Vert \wh{J}_n - J_n \Vert\, \Vert \wh{\beta}_n - \beta_n \Vert = o_p(n^{-1/2})$ and the result follows from \ref{HypCA}, that is from $J_n \rightarrow J$. \qed

\subsection{Proof of Theorem \ref{ThmNormAlpha}.}
\label{SecProofThmAlpha}

Taking the first line of the convergence in Lemma \ref{LemCvgVarTheta},
\begin{equation}
\label{eq_NormVtet1}
	\sqrt{n^{1+\alpha}} \, \big( \wh{v}_n(\alpha) - \rho_n(\alpha) \big) \cvgd \cN\big( 0,\, c\, \ell^{-1} \big)
\end{equation}
where we recall that $\wh{v}_n(\alpha)$ is the first component of $\whvn(\alpha)$, $\rho_n(\alpha) = 1 - c\,n^{-\alpha}$ and $\ell > 0$ is given in Lemma \ref{LemCvgSn}. Now define the sequence of $\mathcal{C}^1$-mappings from $]0,\, 1[$ to $\dR$ given by
\begin{equation*}
\forall\, n > 1, \hsp g_n(x) = \frac{\ln c - \ln\left( 1 - x \right)}{\ln n}
\end{equation*}
so that $\alpha = g_n(\rho_n(\alpha))$ and $\wh{\alpha}_n(\alpha) = g_n(\wh{v}_n(\alpha))$. Using Taylor's expansion, there is a sequence $(\vartheta^*_{n})$ such that
\begin{equation*}
\wh{\alpha}_n(\alpha) = \alpha + \big( \wh{v}_n(\alpha) - \rho_n(\alpha) \big)\, g_n^{\prime}(\vartheta^*_{n})
\end{equation*}
with, for all $n > 1$, $(\rho_n(\alpha) \wedge \wh{v}_n(\alpha)) < \vartheta^*_{n} < (\rho_n(\alpha) \vee \wh{v}_n(\alpha))$. So,
\begin{equation}
\label{eq_DecompAlpha}
		(\ln n)\, \sqrt{n^{1-\alpha}}\, \left( \wh{\alpha}_n(\alpha) - \alpha \right) =  \sqrt{n^{1+\alpha}}\,\big( \wh{v}_n(\alpha) - \rho_n(\alpha) \big) \, \frac{n^{-\alpha}}{1-\vartheta^*_{n}}.
\end{equation}
We also note that
\begin{equation*}
n^{\alpha}\, \big( 1 - \rho_n(\alpha) \vee \wh{v}_n(\alpha) \big) ~<~  n^{\alpha}\, \big(1 - \vartheta^*_{n} \big) ~<~ n^{\alpha}\, \big( 1- \rho_n(\alpha) \wedge \wh{v}_n(\alpha) \big)
\end{equation*}
and
\begin{eqnarray*}
n^{\alpha}\, \big( 1 - \wh{v}_n(\alpha) \big) & = & n^{\alpha}\, \big( 1 - \rho_n(\alpha) \big) + n^{\alpha}\, \big( \rho_n(\alpha) - \wh{v}_n(\alpha) \big) \\
 & = & c + O_p\big( \sqrt{n^{\alpha-1}} \big) \cvgp c
\end{eqnarray*}
from Lemma \ref{LemCvgVarTheta} and since $\alpha < 1$. That clearly implies that
\begin{equation}
\label{eq_cvgVtetStar}
n^{\alpha}\, \big(1 - \vartheta^*_{n} \big) \cvgp c.
\end{equation}
The result follows from \eqref{eq_NormVtet1}, \eqref{eq_DecompAlpha} and \eqref{eq_cvgVtetStar} together with Slutsky's lemma. \qed

\subsection{Proof of Theorem \ref{ThmTest}.}
\label{SecProofThmTest}
First, a direct computation of the entries of $P^{-1}$ shows that
\begin{equation*}
    \pi_{11} = \frac{1}{\prod_{j=2}^p (1 - \lambda_j)} \hsp (= 1 ~\text{if $p=1$}),
\end{equation*}
see \eqref{DefPi11} with $\lambda_1 = 1$. We know by Rem. 2.2 in \citet{BadreauProia2023} that the OLS of $\theta_n$ defined in \eqref{OLS} is consistent for $\theta$, so obviously
\begin{equation*}
\wh{A}_n \cvgp A
\end{equation*}
also holds. Since the eigenvalues of a square matrix are continuous functions of its entries, see e.g. Thm. 2.4.9.2 of \citet{HornJohnson92}, it follows by the continuous mapping theorem that
\begin{equation*}
\wh{\pi}_n \cvgp \pi_{11}.
\end{equation*}
Then, taking for $\alpha_0$ the `true' value $\alpha$ and using Slutsky's lemma, \eqref{CvgZnH0} is a straightforward consequence of Theorem \ref{ThmNormAlpha}. Now suppose that $\alpha/2 < \alpha_0 < \alpha$. Coming back to \eqref{eq_DecompAlpha} with $\alpha_0$, through the reasoning of the previous proof, there is a sequence $(\vartheta^{**}_{n})$ such that
\begin{equation}
\label{eq_DecompAlpha0}
		(\ln n)\, \sqrt{n^{1-\alpha_0}}\, \left( \wh{\alpha}_n(\alpha_0) - \alpha_0 \right) = \sqrt{n^{1+\alpha_0}}\, \big( \wh{v}_n(\alpha_0) - \rho_n(\alpha_0) \big)\, \frac{n^{-\alpha_0}}{1-\vartheta^{**}_{n}}
\end{equation}
where $(\rho_n(\alpha_0) \wedge \wh{v}_n(\alpha_0)) < \vartheta^{**}_{n} < (\rho_n(\alpha_0) \vee \wh{v}_n(\alpha_0))$. Using again the decomposition \eqref{DecompVartheta}, we first obtain with \eqref{DecompCvgR2},
\begin{equation*}
\big\vert n^{\alpha_0} \, e_p\T \tilde{R}_{n,\,2} \big\vert = O_p\big( \sqrt{n^{2 \alpha_0 - 1 - \alpha}} \big) \cvgp 0
\end{equation*}
since $2 \alpha_0 < \alpha + 1$. Then,
\begin{equation*}
n^{\alpha_0} \, e_p\T \left(\tilde{R}_{n,\,1} + \tilde{R}_{n,\,3}\right) = c \left(1 - n^{\alpha_0 - \alpha}\right) (1 - u_n) \cvgp c\, ( 1 - s_{\beta})
\end{equation*}
where $u_n$ has been introduced in \eqref{DecompRn13} and its limit behavior established in \eqref{CvgUn}. Hence,
\begin{equation*}
    n^{\alpha_0}\, \big( 1 - \wh{v}_n(\alpha_0) \big) = c + n^{\alpha_0}\, \big( \rho_n(\alpha_0) - \wh{v}_n(\alpha_0) \big) \cvgp c\, s_{\beta}.
\end{equation*}
This is sufficient to get that, for any $\alpha/2 < \alpha_0 < \alpha$,
\begin{equation}
\label{eq_DivVthetStar}
\liminf_{n\, \rightarrow\, +\infty}\, \frac{n^{-\alpha_0}}{\vert 1-\vartheta^{**}_{n} \vert} ~\geq~ \frac{1}{c\, (1 \vee \vert s_{\beta} \vert)} ~>~ 0.
\end{equation}
Finally, \eqref{DivVarthetAlpha0} entails
\begin{equation}
\label{eq_DivVtet1Alpha0}
	\left\vert \sqrt{n^{1+\alpha_0}} \, \big(\wh{v}_n(\alpha_0) - \rho_n(\alpha_0) \big)\right\vert \cvgp + \infty.
\end{equation}
The combination of \eqref{eq_DecompAlpha0}, \eqref{eq_DivVthetStar} and \eqref{eq_DivVtet1Alpha0} leads to \eqref{CvgZnH1}, which concludes the proof.
\qed

\subsection{The case of a negative unit root}
\label{SecNegRoot}
Our results have all been established for $\lambda_1 = 1$ but they remain true in a context of a negative unit root, at the cost of some slight adjustments. In what follows, it is thus implied that $\lambda_1 = -1$ so that we may now consider, $\lambda_{n,\, 1} = \lambda_1\, \rho_n(\alpha) = -\rho_n(\alpha)$. \correc{The immediate consequence is that, in the definition of the latent variables $V_{n,\, k}$ in \eqref{ModHiera}, $V_{n,\, k}(\alpha) = X_{n,\, k} + \rho_n(\alpha)\, X_{n,\, k-1}$.} So, in the proof of Lemma \ref{LemCvgTn}, $\rho_n(\alpha)$ must be replaced by $-\rho_n(\alpha)$ in \eqref{DefNonDiagBlocks} and \eqref{eq_XV}, and in particular \correc{that leads to
\begin{eqnarray*}
C_n^{\frac{1}{2}}(\alpha)\, P_n\T \left( I_p + \rho_n(\alpha) A_n\T \right) & = & C_n^{\frac{1}{2}}(\alpha) \left( I_p + \rho_n(\alpha) D_n \right) P_n\T \\
 & \cvgp & \diag\left( 0, (1+\lambda_i)_{2\, \leq\, i\, \leq\, p} \right) P\T \eqdef \Delta_p\, P\T
\end{eqnarray*}
in place of \eqref{DefDiagDelta}.} We may also note that $c$ merely becomes $-c$ in \eqref{DefPsiPoint}, \eqref{DecompTn} and \eqref{DecompVartheta}. The sequence $(u_n)$ defined in \eqref{DecompRn13} \correc{now satisfies 
\begin{equation*}
    u_n \cvgp e_p\T \, \Upsilon_p^{-1} \, \sigma^2 \, \ell \, \begin{pmatrix}
    0 & -1 & \hdots & (-1)^{p-1} \\
    \vdots & 0 & \hdots & 0 \\
    \vdots & \vdots & & \vdots \\
    0 & 0 & \hdots & 0 \\
    \end{pmatrix} \, \begin{pmatrix}
    -1 \\
    \vdots \\
    \beta_i \\
    \vdots \\
    \end{pmatrix} ~=~ \sum_{i=1}^{p-1} (-1)^i \beta_i \eqdef s_{\beta}
\end{equation*}
where $1 - s_\beta = \Theta_\beta(-1) \neq 0$.} It is interesting to note that, for $\lambda_1 = \pm 1$, we have in general terms
\begin{equation*}
1 - u_n \cvgp \Theta_\beta(\lambda_1)
\end{equation*}
where $\Theta_\beta(z) = 1 - \beta_1 z - \ldots - \beta_{p-1} z^{p-1}$ has no zero on the unit circle, by stability: the sign of the asymptotic behavior of $Z_n$ under $\cH_1$ thus depends on the sign of $\Theta_\beta(\lambda_1) \neq 0$. Finally, by considering the $\mathcal{C}^1$-mappings from $]\!-1,\, 0[$ to $\dR$ given by
\begin{equation*}
\forall\, n > 1, \hsp g_n(x) = \frac{\ln c - \ln\left( 1 + x \right)}{\ln n},
\end{equation*}
one can see that $\alpha = g_n(-\rho_n(\alpha))$ and $\wh{\alpha}_n(\alpha) = g_n(\wh{v}_n(\alpha))$ as defined in \eqref{EstAlpha}. Theorems \ref{ThmNormAlpha} and \ref{ThmTest} easily follow.

\bigskip

\noindent \textbf{Acknowledgements}. \correc{The authors warmly thank the Associate Editor and the two anonymous Reviewers for their suggestions and constructive comments which helped to improve this work.} This research benefited from the support of the ANR project `Efficient inference for large and high-frequency data' (ANR-21-CE40-0021).

\bibliographystyle{plainnat}
\bibliography{QUnitRoot}

\end{document}